\documentclass[12pt]{amsart}
\usepackage[utf8]{inputenc}
\usepackage[margin=2.5cm]{geometry}
\usepackage{amsmath}
\usepackage{amsfonts}
\usepackage{amssymb}
\usepackage{amsthm}
\usepackage{tikz-cd}
\usepackage{mathtools}
\usepackage{outlines}
\usepackage{hyperref}
\usepackage{verbatim}

\counterwithin{equation}{section}

\theoremstyle{plain}
\newtheorem*{thm*}{Theorem}
\newtheorem{thm}{Theorem}[section]
\newtheorem{prop}[thm]{Proposition}
\newtheorem{lem}[thm]{Lemma}
\newtheorem{cor}[thm]{Corollary}

\theoremstyle{definition}
\newtheorem{qstn}[thm]{Question}
\newtheorem{dfn}[thm]{Definition}

\newtheorem{rem}[thm]{Remark}
\newtheorem{ntt}[thm]{Notation}
\newtheorem{case}{Case}

\makeatletter
    \@addtoreset{case}{thm}
    \@addtoreset{case}{prop}
    \@addtoreset{case}{lem}
    \@addtoreset{case}{cor}
\makeatother

\newcommand{\ZZ}{\mathbb{Z}}

\newcommand{\NN}{\mathbb{N}}

\newcommand{\kk}{\Bbbk}

\newcommand{\II}{\mathbb{I}}

\newcommand{\mc}{\mathcal}

\newcommand{\mb}{\mathbb}

\newcommand{\ang}[1]{\left\langle #1 \right\rangle}

\newcommand{\nonzero}{\setminus \{0\}}

\DeclareMathOperator{\im}{Im}

\DeclareMathOperator{\GKdim}{GKdim}

\DeclareMathOperator{\id}{id}

\DeclareMathOperator{\spn}{span}

\newcommand{\ndo}{\underline{\operatorname{end}}}
\newcommand{\aut}{\underline{\operatorname{aut}}}

\newcommand{\qdet}{\operatorname{det}_q}
\newcommand{\invqdet}{\operatorname{det}_{q^{-1}}}

\DeclareMathOperator{\Inv}{Inv}

\newcommand{\surj}{\twoheadrightarrow}

\newcommand{\wt}{\widetilde}

\makeatletter
    \newcommand{\extp}{\@ifnextchar^\@extp{\@extp^{\,}}}
    \def\@extp^#1{\mathop{\bigwedge\nolimits^{\!#1}}}
\makeatother

\newcommand\restr[2]{{
  \left.\kern-\nulldelimiterspace 
  #1 
  \vphantom{\big|}
  \right|_{#2}
  }}

\newcommand{\etalchar}[1]{$^{#1}$}

\title{Coactions of cocommutative Hopf algebras on skew polynomial rings}
\author{Lucas Buzaglo}
\author{Daniel Rogalski}
\date{}

\address{Department of Mathematics, UC San Diego, La Jolla, CA 92093-0112, USA}
\email{\href{mailto:lbuzaglo@ucsd.edu}{lbuzaglo@ucsd.edu}, \href{mailto:drogalski@ucsd.edu}{drogalski@ucsd.edu}}


\begin{document}

\begin{abstract}
    We classify the cocommutative Hopf algebras which coact inner-faithfully on (one-parameter) skew polynomial rings $A_q(n) = \kk\ang{x_1,\dots,x_n}/(x_j x_i - q x_i x_j \mid i < j)$ for $n = 2$ and $3$. As a direct corollary, we obtain a classification of group gradings on two- and three-variable skew polynomial rings, recovering a result of Crawford in the two-variable case. Our results are achieved via Manin's universal coacting Hopf algebra construction, often denoted $\aut(A_q(n))$, by classifying all its cocommutative quotients. We therefore also give an explicit presentation of $\aut(A_q(n))$ for arbitrary $q \in \kk^*$ and $n \in \NN$.
\end{abstract}

\maketitle

\section{Introduction}

The main motivation of this paper (as well as \cite{BuzagloRogalski}) is to understand \emph{quantum group symmetries}, a generalization of the classical work on group actions on polynomial rings. In other words, we are interested in Hopf algebras which (co)act on Artin--Schelter (AS) regular algebras. The most fundamental examples of Hopf algebras are group algebras (which act by automorphisms) and universal enveloping algebras of Lie algebras (which act by derivations). Together, these two classes of Hopf algebras essentially encompass all actions by cocommutative Hopf algebras, by the Cartier--Kostant--Gabriel theorem. Much work has been done on actions by finite-dimensional cocommutative (or more generally, semisimple) Hopf algebras, although many open questions remain. See, for example, \cite{CohenWestreich, KirkmanKuzmanovichZhang, EtingofWalton, ChanWaltonZhang, ChanWaltonWangZhang, ChanKirkmanWaltonZhang} and references therein.

In this paper, we study actions by dual group algebras, which are commutative Hopf algebras. If $G$ is a finite group, an action by the dual group algebra $(\kk G)^*$ on a $\kk$-algebra $A$ is equivalent to a $G$-grading on $A$ \cite[Example 1.6.7]{Montgomery}, where $\kk$ is a field. The invariant ring under the $(\kk G)^*$-action is then easily seen to be $A_e$, the identity component of $A$ under the $G$-grading. We study dual group actions by viewing a $(\kk G)^*$-action as a \emph{coaction} by $\kk G$ (see Definition \ref{dfn:coaction}).

Of main interest are coactions by noncommutative Hopf algebras, called \emph{genuine} coactions in \cite{EtingofGoswamiMandalWalton}; see also \cite{KramerOni}. Indeed, a Hopf algebra is said to admit \emph{no quantum symmetry} when it does not admit inner-faithful coactions by noncommutative Hopf algebras. Coactions on two-dimensional AS regular algebras are reasonably well understood -- see, for example, \cite{WaltonWang, RaedscheldersVanDenBergh2}. Indeed, the classification of $\kk G$-coactions (or $G$-gradings) in this case is complete: the only two-dimensional AS regular algebra which admits an inner-faithful coaction by the group algebra of a nonabelian group is the quantum plane $\kk\ang{x,y}/(xy + yx)$, graded by the nonabelian group $\Gamma \coloneqq \ang{f,g \mid f^2 = g^2}$ \cite[Theorem 1.1]{Crawford}. Some work has been done on coactions on higher-dimensional AS regular algebras \cite{KirkmanKuzmanovichZhang2, GoetzKirkmanMooreVashaw}, but no such classification exists in dimensions three and higher.

In this paper, we focus on three-dimensional AS regular algebras, with the goal of generalizing Crawford's result classifying group gradings on two-dimensional AS regular algebras. One advantage of our approach through coactions it that we can use Manin's \emph{universal coacting Hopf algebra} construction \cite{ManinBook}: given a quadratic algebra $A$, there is a Hopf algebra $\aut(A)$ which (right) coacts on it in a universal way, preserving the $\NN$-grading of $A$ (see Definition \ref{dfn:Manin}). If we assume that $A$ has a \emph{faithful} $G$-grading (see Definition \ref{dfn:grading}) which refines its $\NN$-grading, then $\aut(A)$ surjects onto $\kk G$. Therefore, we can classify all the group gradings of $A$ by classifying the group algebras which arise as quotients of $\aut(A)$.

Notice that a Hopf algebra is a group algebra if and only if it is cocommutative and is spanned by its group-like elements. Given this observation, our strategy is as follows.
\begin{enumerate}
    \item Classify the cocommutative Hopf quotients of $\aut(A)$.\label{obj:intro cocommutative}
    \item Given a cocommutative Hopf quotient of $\aut(A)$, determine when it is spanned by its group-like elements.\label{obj:intro group algebra}
\end{enumerate}
Step \eqref{obj:intro group algebra} is straightforward, so the main difficulty lies in step \eqref{obj:intro cocommutative}. To find cocommutative quotients of $\aut(A)$, we use the well-known fact that a cocommutative Hopf algebra is involutive, meaning its antipode has order two. Therefore, step \eqref{obj:intro cocommutative} is split into two parts: first, we \emph{involutize} $\aut(A)$ by taking the quotient by the Hopf ideal $(S^2(x) - x \mid x \in \aut(A))$ (see Definition \ref{dfn:involutization}), and then we find cocommutative quotients of the involutization of $\aut(A)$.

For simplicity, we restrict our attention to the \emph{(one-parameter) skew polynomial rings} over an algebraically closed field $\kk$ of characteristic zero:
$$A_q = A_q(n) \coloneqq \frac{\kk\ang{x_1,\dots,x_n}}{(x_j x_i - q x_i x_j \mid i < j)},$$
where $q \in \kk^*$. First, we describe the universal coacting Hopf algebra $\aut(A_q)$. Although this Hopf algebra has been studied before, it is difficult to find the explicit presentation of $\aut(A_q)$ that we need in the literature for arbitrary $q \in \kk^*$ and an arbitrary number of variables. See, for example, \cite[Example 9.6]{ManinBook}, \cite[Lemma 5.1]{ChanWaltonZhang}, \cite[Appendix A]{RaedscheldersVanDenBergh}, \cite{RaedscheldersVanDenBergh2}, and \cite{ChirvasituWaltonWang}. Therefore, our first main result is a description of $\aut(A_q)$. In the theorem below, the \emph{quantum determinant} of a $k \times k$ matrix $P = (p_{ij})$ is defined by
$$\qdet(P) \coloneqq \sum_{\sigma \in S_k} (-q)^{-I(\sigma)} p_{1 \sigma(1)} \dots p_{k\sigma(k)},$$
where $I(\sigma)$ is the inversion number of $\sigma$, and $P_{\hat\imath \hat\jmath}$ is the $(k - 1) \times (k - 1)$ matrix obtained by removing row $i$ and column $j$ from the matrix $P$.
 
\begin{thm}[Theorem \ref{thm:description of universal Hopf algebra}]
    Let $q \in \kk^*$. The Hopf algebra $\aut(A_q)$ is generated by $x_{ij}$ for $1 \leq i,j \leq n$ and $D^{-1}$, subject to the following relations:
    \begin{gather*}
        x_{ik}x_{ij} = q x_{ij}x_{ik} \quad (k > j), \\
        ad - q^{-1}bc = \qdet\begin{pmatrix}
            a & b \\
            c & d
        \end{pmatrix} = da - q cb, \\
        \sum_{k = 1}^n (-q)^{k - i} \qdet(X_{\hat k \hat\imath}) D^{-1} x_{k j} = \delta_{ij}, \\
        D^{-1} D = D D^{-1} = 1,
    \end{gather*}
    for all $i,j,k \in \{1,\dots,n\}$, and all $2 \times 2$ sub-matrices $\begin{pmatrix}
        a & b \\
        c & d
    \end{pmatrix}$ of the matrix $X \coloneqq (x_{ij})$, where $D \coloneqq \qdet(X)$ is the quantum determinant of $X$. The Hopf structure of $\aut(A_q)$ is given by
    \begin{align*}
        \Delta(x_{ij}) &= \sum_{k = 1}^n x_{ik} \otimes x_{kj}, \\
        \varepsilon(x_{ij}) &= \delta_{ij}, \\
        S(x_{ij}) &= (-q)^{j - i} \qdet(X_{\hat{\jmath}\hat{\imath}})D^{-1}.
    \end{align*}
\end{thm}

We then proceed to classify the cocommutative Hopf algebras which right coact on $A_q(2)$ inner-faithfully by classifying the cocommutative quotients of $\aut(A_q(2))$.

\begin{thm}[Theorem \ref{thm:cocommutative quotient 2 variables}]\label{thm:intro cocommutative 2 variables}
    Let $q \in \kk^*$. If $q \neq -1$, then the only cocommutative Hopf algebras which right coact on $A_q(2)$ inner-faithfully are the Hopf quotients of $\kk\ZZ^2$ and $\mathcal{A}(0,q^{\pm 1})$ (see Definition \ref{dfn:Goodearl-Zhang Hopf algebra}). If $q = -1$, we get an additional coaction by (Hopf quotients of) $\kk \Gamma$.
\end{thm}

As a consequence of Theorem \ref{thm:intro cocommutative 2 variables}, we recover Crawford's classification of group gradings on two-variable skew polynomial rings \cite[Theorem 1.1]{Crawford}.  The formulas for the coactions are omitted here, as they are clear from the proof, which constructs each of these Hopf algebras coacting on $A_q(2)$ as an explicit quotient of the universal involutive coacting Hopf algebra $\aut(A_q(2))$.

We then prove a similar result for three-variable skew polynomial rings. Note that we have to exclude the case $q = \pm 1$ due to the computational difficulty of this situation.

\begin{thm}[Theorem \ref{thm:generic}]\label{thm:intro cocommutative 3 variables}
    Let $q \in \kk^* \setminus \{\pm 1\}$. Then the only cocommutative Hopf algebras which right coact on $A_q(3)$ inner-faithfully are the Hopf quotients of $\kk\ZZ^3, \mathcal{B}_{q^{\pm 1}}$, and $\mathcal{C}_{q^{\pm 1}}$ (see Definition \ref{dfn:Bq and Cq}). Consequently, $A_q(3)$ does not have a faithful grading by a nonabelian group.
\end{thm}

It is also clear that any grading of the commutative polynomial ring $A_1(3)$ must be by a factor group of $\mb{Z}^3$, so the preceding result implies that only when $q = -1$ can $A_q(3)$ possibly be graded by a nonabelian group, similarly as for $A_q(2)$.  Conversely, gradings by nonabelian exist for $q = -1$, because writing $A_{-1}(3) \cong A_{-1}(2)[z; \sigma]$ as an Ore extension, one easily sees that $A_{-1}(3)$ is graded by $\Gamma \times \mb{Z}$.  

In the upcoming paper \cite{BuzagloRogalski}, we will study group gradings of more general regular algebras of dimension $3$ by different methods, involving an explicit analysis of superpotentials.  This will allow us to show that the only gradings of $A_{-1}(3)$ by nonabelian groups $G$ are the obvious ones where $G$ is a factor of $\Gamma \times \mb{Z}$. The methods of \cite{BuzagloRogalski} will not give us information on coactions by more general cocommutative Hopf algebras, however.

The paper is organized as follows: in Section \ref{sec:preliminaries}, we give the necessary definitions and basic results from the literature that we use throughout the paper. In Sections \ref{sec:Frobenius} and \ref{sec:Manin Hopf algebra of Aq}, we construct a presentation of the universal coacting Hopf algebra $\aut(A_q)$, and we describe its involutization in Section \ref{sec:involutization}. Next, we classify cocommutative coactions on two- and three-variable skew polynomial rings in Sections \ref{sec:two variables} and \ref{sec:three variables}. Finally, we present some applications of our results and some open questions in Section \ref{sec:applications}.

\subsection*{Acknowledgments}

We thank Ellen Kirkman, Akira Masuoka, Susan Montgomery, Chelsea Walton, and James Zhang for interesting discussions. The first-named author is supported by an AMS--Simons travel grant.

\section*{Conventions}

Throughout, we work over an algebraically closed field $\kk$ of characteristic zero. The notations $\Delta$, $\varepsilon$, and $S$ denote the coproduct, counit, and antipode of any Hopf algebra. If needed for clarity, we use the notation $\Delta_H$, $\varepsilon_H$, and $S_H$ for the coproduct, counit, and antipode of a Hopf algebra $H$. The notations $G(H)$ and $P(H)$ denote the group of group-like elements of $H$ and the Lie algebra of primitive elements of $H$, respectively.

\section{Preliminaries}\label{sec:preliminaries}

We start by outlining our goals and recalling the required definitions and algebras of interest for this paper.

\subsection{Group gradings and coactions of group algebras}

First, we recall the definition of a group grading.

\begin{dfn}\label{dfn:grading}
    We say that a $\kk$-algebra $A$ is \emph{graded} by a group $G$, or \emph{$G$-graded}, if $A$ can be decomposed into abelian groups as
    $$A = \bigoplus_{g \in G} A_g, \quad \text{where } A_g A_h \subseteq A_{gh},$$
    for all $g,h \in G$.

    If $A$ is $G$-graded, we say that the $G$-grading is \emph{faithful} if $\{g \in G \mid A_g \neq 0\}$ generates the group $G$.
\end{dfn}

We will always assume that group gradings are faithful, since this means that the algebra cannot be graded by a proper subgroup of the grading group.

\begin{rem}
    The notion of a faithful grading from Definition \ref{dfn:grading} matches the terminology used in \cite{Dascalescu}. However, other authors prefer to call such a grading \emph{connected} (see, for example, \cite[Proposition 2.4]{CibilsRedondoSolotar}).
\end{rem}

It is well-known that a grading of a $\kk$-algebra $A$ by a group $G$ is equivalent to a (left or right) coaction by the group algebra $\kk G$ \cite[Example 1.6.7]{Montgomery}. We define the notion of a coaction next.

\begin{dfn}\label{dfn:coaction}
    Let $H$ be a Hopf algebra. A vector space $M$ is a \emph{right comodule} for $H$, or that $H$ \emph{right coacts} on $M$, if there is a map
    $$\rho \colon M \to M \otimes H$$
    that is compatible with the Hopf algebra structure of $H$. In other words, $\rho$ is \emph{co-associative} and \emph{co-unital}, meaning the diagrams
    \begin{center}
        \begin{tikzcd}
            M \arrow[r, "\rho"] \arrow[d, "\rho"']       & M \otimes H \arrow[d, "\id_M \otimes \Delta_H"] &  & M \arrow[rd, "-\otimes1"'] \arrow[r, "\rho"] & M \otimes H \arrow[d, "\id_M \otimes \varepsilon_H"] \\
            M \otimes H \arrow[r, "\rho \otimes \id_H"'] & M \otimes H \otimes H                           &  &                                              & M \otimes \kk                                     
        \end{tikzcd}
    \end{center}
    commute. One can similarly define a left coaction as a map $\lambda \colon M \to H \otimes M$ which is co-associative and co-unital.

    If $A$ is a $\kk$-algebra, we that $A$ is a \emph{right $H$-comodule algebra}, or that $H$ \emph{right coacts} on $A$, if $A$ is a right $H$-comodule, and the comodule map $\rho \colon A \to A \otimes H$ is an algebra homomorphism. A left comodule algebra is defined similarly.
\end{dfn}

We now describe how a coaction of a group algebra $\kk G$ on an algebra $A$ gives rise to a $G$-grading of $A$. To that end, suppose $A$ is a right $\kk G$-comodule algebra via the map $\rho \colon A \to A \otimes \kk G$. Given $a \in A$, write
$$\rho(a) = \sum_{g \in G} a_g \otimes g$$
where $a_g \in A$. For $g \in G$, define
$$A_g \coloneqq \{a_g \mid a \in A\}.$$
It is easy to show using the co-associativity and co-unitality of $\rho$ that $A = \bigoplus_{g \in G} A_g$, and we can further use that $\rho$ is an algebra homomorphism to deduce that $A_g A_h \subseteq A_{gh}$.

Since we always assume that group gradings are faithful, this should correspond to some property of the right coaction of the group algebra defined by the grading. This property is \emph{inner-faithfulness}, defined below.

\begin{dfn}
    Let $H$ be a Hopf algebra and let $M$ be a (right) $H$-comodule via the map $\rho \colon M \to M \otimes H$. We say that the $H$-coaction is \emph{inner-faithful} if $\rho(M) \not\subseteq M \otimes H'$ for any proper Hopf subalgebra $H' \subsetneqq H$.
\end{dfn}

In other words, we are interested in inner-faithful coactions of group algebras.

\subsection{Grading skew polynomial rings using Manin's universal coacting Hopf algebra}

Fix $n \geq 2$ and $q \in \kk^*$. Our main goal is to classify (faithful) group gradings of the skew polynomial ring
$$A_q(n) \coloneqq \frac{\kk\ang{x_1,\ldots,x_n}}{(x_j x_i - q x_i x_j \mid j > i)}.$$
We usually drop the $(n)$ if the value of $n$ is clear and simply write $A_q$ instead of $A_q(n)$.

We will classify the gradings by classifying the possible (inner-faithful) coactions of group algebras on $A_q$. We only consider gradings of $A_q$ which are compatible with its natural $\NN$-grading, where each $x_i$ is homogeneous of degree 1. In terms of coactions, this means that we require the map $A_q \to A_q \otimes \kk G$ to respect the $\NN$-grading of $A_q$. In other words, $x_i$ maps to $\sum_{j = 1}^n x_j \otimes g_{ji}$ for some $g_{ji} \in \kk G$.

To achieve our classification, we will use the following construction: there is a universal right coacting Hopf algebra for $A_q$, denoted $\aut^r(A_q)$ \cite{ManinBook}. We define this Hopf algebra now, but we postpone its explicit construction to Sections \ref{sec:Frobenius} and \ref{sec:Manin Hopf algebra of Aq}.

\begin{dfn}\label{dfn:Manin}
    A $\kk$-algebra $A$ is \emph{quadratic} if it can be presented as $A = TV/(R)$, where $V$ is a finite-dimensional vector space and $R \subseteq V \otimes V$. A quadratic algebra $A = TV/(R)$ is automatically $\NN$-graded as $A = \bigoplus_{k \in \NN} A_k$, where $A_k$ is the image of $V^{\otimes k}$ in $A$.

    Given a quadratic algebra $A$, the \emph{universal right coacting Hopf algebra} of $A$, denoted $\aut^r(A)$, is the Hopf algebra which right coacts on $A$ via a map $\rho \colon A \to A \otimes \aut^r(A)$, compatible with its natural $\NN$-grading, satisfying the following universal property: if $H$ is a Hopf algebra which right coacts on $A$ via a map $\rho' \colon A \to A \otimes H$, also compatible with its natural $\NN$-grading, then there is a unique Hopf algebra homomorphism $\phi \colon \aut^r(A) \to H$ such that the diagram
    \begin{center}
        \begin{tikzcd}
            & A \otimes \aut^r(A) \arrow[d, dashed, "\id_A \otimes \phi"]  \\
            A \arrow[ru, "\rho"] \arrow[r, "\rho'"] & A \otimes H
        \end{tikzcd}
    \end{center}
    commutes. There is also a \emph{universal left coacting Hopf algebra} $\aut^\ell(A)$ defined analogously.
\end{dfn}
While we have emphasized coactions by Hopf algebras here, one can make the same definitions for coactions by bialgebras. There is a bialgebra $\ndo^r(A)$ which right coacts on $A$, compatible with the grading, and which is universal in the same sense for right coactions by bialgebras that are compatible with the grading.  Similarly, there is a universal left coacting bialgebra $\ndo^{\ell}(A)$.  We will need the bialgebra construction only in passing, as an intermediate step in constructing $\aut^r(A)$.

If we have a $G$-grading on $A_q$, then by the universal property there exists a unique Hopf algebra homomorphism $\aut^r(A_q) \to \kk G$. This homomorphism must be surjective, thanks to the assumption that the $G$-grading on $A_q$ is faithful. We therefore approach the classification of group gradings on $A_q$ by classifying Hopf algebra homomorphisms from $\aut^r(A_q)$ to group algebras.

\subsection{Observations about group algebras}\label{subsec:observations}

A Hopf algebra $H$ is group algebra if and only if $H$ is cocommutative and is spanned by its group-like elements. Given this observation, our strategy to achieve the classification of group gradings of $A_q$ is the following:
\begin{enumerate}
    \item Classify Hopf ideals $I$ of $\aut^r(A_q)$ such that $\aut^r(A_q)/I$ is cocommutative.\label{obj:cocommutative quotient}
    \item For $I$ as above, determine when $\aut^r(A_q)/I$ is spanned by its group-like elements.\label{obj:group-likes}
\end{enumerate}
To aid with \eqref{obj:cocommutative quotient}, we use the well-known property that cocommutative Hopf algebras are involutive (defined below).

\begin{dfn}
    A Hopf algebra $H$ is \emph{involutive} if $S_H^2 = \id_H$.
\end{dfn}

\begin{prop}[{\cite[Corollary 1.5.12]{Montgomery}}]
    Cocommutative Hopf algebras are involutive.
\end{prop}

Therefore, we make the following observation: let $I$ be a Hopf ideal of $\aut^r(A_q)$ such that $H \coloneqq \aut^r(A_q)/I$ is cocommutative. Then it must be the case that $S_{H}^2 = \id_{H}$. Hence, to find cocommutative quotients of $\aut^r(A_q)$, we first ``involutize'' $\aut^r(A_q)$ (see Definition \ref{dfn:involutization}), and then we further specialize to cocommutative quotients. Once we have found a cocommutative quotient $H$ of $\aut^r(A_q)$, step \eqref{obj:group-likes} is straightforward.

\subsection{Maximal cocommutative quotients}

A Hopf algebra has both a universal commutative quotient (its \emph{abelianization}) and a universal involutive quotient (its \emph{involutization}). However, Hopf algebras do not have universal cocommutative quotients, making the study of cocommutative quotients of Hopf algebras more difficult, but also richer, than studying commutative or involutive quotients. Although universal cocommutative quotients do not exist, we now prove that Hopf algebras still have \emph{maximal} cocommutative quotients.

\begin{prop}\label{prop:maximal cocommutative quotients}
    Let $H$ be a Hopf algebra and let $\varphi \colon H \to C$ be a Hopf algebra homomorphism, where $C$ is cocommutative. Then there exists a cocommutative quotient $K = H/I$ of $H$ (where $I$ is a Hopf ideal of $H$) and a Hopf algebra homomorphism $\psi \colon K \to C$ such that the diagram
    \begin{center}
        \begin{tikzcd}
            & K \arrow[d, dashed, "\psi"]  \\
            A \arrow[ru, "\pi"] \arrow[r, "\varphi"] & C
        \end{tikzcd}
    \end{center}
    commutes, and the cocommutative quotient $K$ is \emph{maximal}, in the following sense: if $I'$ is another Hopf ideal of $H$ strictly contained in $I$, then $H/I'$ is not cocommutative.
\end{prop}

In other words, Proposition \ref{prop:maximal cocommutative quotients} says that every cocommutative quotient of a Hopf algebra is a quotient of one (or more) of the maximal cocommutative quotients. Therefore, one can study cocommutative quotients of a Hopf algebra by classifying the maximal ones. At first, this seems like a potentially impossible task, but we will see later in this paper that we can fully classify the maximal cocommutative quotients of $\aut^r(A_q(n))$ up to isomorphism for $n = 2$ and $3$.

To prove Proposition \ref{prop:maximal cocommutative quotients}, we require the following easy linear algebra result.

\begin{lem}\label{lem:intersection distributes over sum}
    Let $V$ be a vector space and let $\mathcal{B}$ be a basis for $V$. Let $A,B,C$ be subspaces of $V$ which are spanned by subsets of $\mathcal{B}$. Then $A \cap (B + C) = (A \cap B) + (A \cap C)$.
\end{lem}
\begin{proof}
    Let $\mathcal{B}_1$, $\mathcal{B}_2$, and $\mathcal{B}_3$ be bases for $A$, $B$, and $C$, respectively, all of which are contained in $\mathcal{B}$. Then $\mathcal{B}_1 \cap (\mathcal{B}_2 \cup \mathcal{B}_3)$ is a basis for $A \cap (B + C)$. Note that $\mathcal{B}_1 \cap (\mathcal{B}_2 \cup \mathcal{B}_3) = (\mathcal{B}_1 \cap \mathcal{B}_2) \cup (\mathcal{B}_1 \cap \mathcal{B}_3)$, which is a basis for $(A \cap B) + (A \cap C)$. The result follows.
\end{proof}

We now prove Proposition \ref{prop:maximal cocommutative quotients}.

\begin{proof}[Proof of Proposition \ref{prop:maximal cocommutative quotients}]
    Define
    $$S \coloneqq \{J \subseteq \ker(\varphi) \mid J \text{ is a Hopf ideal of } H \text{ and } H/J \text{ is cocommutative}\}.$$
    We apply Zorn's lemma to prove that $S$ contains a minimal element. To that end, let $X \subseteq S$ be a chain in $S$, and define $I \coloneqq \bigcap_{J \in X} J$. We claim that $I$ is a Hopf ideal of $H$. Certainly, $I$ is an ideal of $H$, $\varepsilon(I) = 0$, and $S(I) \subseteq I$ by construction, so it remains to check that $\Delta(I) \subseteq H \otimes I + I \otimes H$.

    Fix $x \in I$, and let $V$ be a finite-dimensional subspace of $H$ such that $\Delta(x) \in V \otimes V$. Letting $J \in X$, we have $x \in J$ by definition, and thus
    $$\Delta(x) \in (H \otimes J + J \otimes H) \cap (V \otimes V).$$
    We claim that $(H \otimes J + J \otimes H) \cap (V \otimes V) = V \otimes (J \cap V) + (J \cap V) \otimes V$, which we will prove by using Lemma \ref{lem:intersection distributes over sum}. To that end, we now construct a basis for $H \otimes H$ which satisfies the hypotheses of Lemma \ref{lem:intersection distributes over sum}. Let $\mathcal{B}_0$ be a basis for $J \cap V$, and let $\mathcal{B}_1$ and $\mathcal{B}_2$ be bases of $V$ and $J$, respectively, which contain $\mathcal{B}_0$. Finally, we let $\mathcal{B}$ be a basis of $H$ containing $\mathcal{B}_1 \cup \mathcal{B}_2$.
    
    Note that $\mathcal{B} \otimes \mathcal{B} = \{u \otimes v \mid u,v \in \mathcal{B}\}$ is a basis of $H \otimes H$. Furthermore, $\mathcal{B}_1 \otimes \mathcal{B}_1$ is a basis of $V \otimes V$, while $\mathcal{B} \otimes \mathcal{B}_2$ and $\mathcal{B}_2 \otimes \mathcal{B}$ are bases of $H \otimes J$ and $J \otimes H$, respectively. Since all of these are subsets of $\mathcal{B} \otimes \mathcal{B}$, it follows by Lemma \ref{lem:intersection distributes over sum} that $(H \otimes J + J \otimes H) \cap (V \otimes V) = V \otimes (J \cap V) + (J \cap V) \otimes V$, as claimed.

    Since $V$ is finite-dimensional, there exists $\mathcal{J} \in X$ such that $J \cap V = \mathcal{J} \cap V$ for all $J \in X$ with $J \subseteq \mathcal{J}$. Note that we then have $I \cap V = \bigcap_{J \in X} J \cap V = \mathcal{J} \cap V$. It follows that
    \begin{multline*}
        \Delta(x) \in \bigcap_{J \in X} \Big(V \otimes (J \cap V) + (J \cap V) \otimes V\Big) = V \otimes (\mathcal{J} \cap V) + (\mathcal{J} \cap V) \otimes V \\
        = V \otimes (I \cap V) + (I \cap V) \otimes V \subseteq H \otimes I + I \otimes H.
    \end{multline*}
    In other words, we have $\Delta(x) \in H \otimes I + I \otimes H$, and thus $\Delta(I) \subseteq H \otimes I + I \otimes H$. Therefore, $I$ is a Hopf ideal, as claimed.

    To conclude the proof, we must show that $H/I$ is cocommutative. Write $\Delta^{\mathrm{op}} \coloneqq \tau \circ \Delta$ for the opposite comultiplication on $H$, where $\tau \colon H \otimes H \to H \otimes H$ is the swap map $\tau(u \otimes v) = v \otimes u$. Since $H/J$ is cocommutative for all $J \in X$, we have
    $$\Delta(h) - \Delta^{\mathrm{op}}(h) \in \bigcap_{J \in X} \big(H \otimes J + J \otimes H\Big)$$
    for all $h \in H$. Proceeding as above, we can show that $\Delta(h) - \Delta^{\mathrm{op}}(h) \in H \otimes I + I \otimes H$, which proves that $H/I$ is cocommutative. The result follows.
\end{proof}

\section{The universal Hopf algebra of a Frobenius algebra}\label{sec:Frobenius}

Instead of directly describing the construction of the universal coacting Hopf algebra $\aut^r(A_q)$, it is easier to construct $\aut^\ell(A_q^!)$, where $A_q^!$ is the \emph{quadratic dual} of $A_q$. It is well-known that $\aut^\ell(A_q^!)$ and $\aut^r(A_q)$ are isomorphic \cite[Theorem 6.10]{ManinBook}.

\begin{dfn}
    The \emph{quadratic dual} of a quadratic algebra $A$, denoted $A^!$, is defined by
    $$A^! \coloneqq TV^*/(R^\perp),$$
    where $R^\perp \subseteq V^* \otimes V^*$ consists of the elements of $V^* \otimes V^*$ vanishing on $R$, in other words,
    $$R^\perp \coloneqq \{\alpha \in V^* \otimes V^* \mid \alpha(R) = 0\}.$$
\end{dfn}

The quadratic dual of $A_q$ is
$$A^!_q \coloneqq \frac{\kk\ang{t_1,\ldots,t_n}}{(t_i^2, t_j t_i + q^{-1} t_i t_j \mid j > i)},$$
which is a \emph{quantum Grassmann algebra} of dimension $n$, a special type of \emph{Frobenius algebra}. We define these notions next.

\begin{dfn}
    A quadratic algebra $A$ is said to be a \emph{Frobenius algebra} of dimension $n$ if
    \begin{enumerate}
        \item $\dim(A_n) = 1$.
        \item $A_k = 0$ for $k > n$.
        \item For all $k$, the multiplication map $A_k \otimes A_{n - k} \to A_n$ is a perfect duality.
    \end{enumerate}
    The algebra $A$ is called a \emph{quantum Grassmann algebra} if, in addition,
    \begin{itemize}
        \item[(4)] $\dim A_k = \binom{n}{k}$.
    \end{itemize}
\end{dfn}

Let $A = TV/(R)$ be a Frobenius algebra of dimension $n$. We now describe the construction of $\aut^\ell(A)$, which will later allow us to construct $\aut^r(A_q)$. The content of this section is well-known (see, for example, \cite[Chapter 9]{ManinBook} and \cite[Section 3]{ChanWaltonZhang}), but we still opt to present this construction in detail to fix our notation and to maximize clarity.

Before we begin describing $\aut^\ell(A)$, we mention the properties that $\aut^\ell(A)$ must satisfy: since $A$ is a graded left $\aut^\ell(A)$-comodule algebra via a map $\lambda \colon A \to \aut^\ell(A) \otimes A$, we have:
\begin{enumerate}
    \item $\lambda$ is co-associative and co-unital.
    \item $\lambda$ is a $\kk$-algebra homomorphism.
    \item $\lambda$ respects the natural $\NN$-grading of $A$, meaning $\lambda(A_1) \subseteq \aut^\ell(A) \otimes A_1$.
\end{enumerate}
Furthermore, $\aut^\ell(A)$ is the universal Hopf algebra satisfying the above properties. This means that we should only impose relations on $\aut^\ell(A)$ which are forced by the above conditions. Doing so will completely describe $\aut^\ell(A)$.

We now begin the construction of $\aut^\ell(A)$. First, since $\dim(A_n) = 1$, there exists $D \in \aut^\ell(A)$ such that $\lambda(w) = D \otimes w$ for all $w \in A_n$. The element $D \in \aut^\ell(A)$ is called the \emph{quantum determinant} of $\aut^\ell(A)$. It is easy to see that $D$ is group-like.

\begin{lem}
    The quantum determinant $D$ is group-like. Consequently, $D$ has an inverse $D^{-1} \in \aut^\ell(A)$.
\end{lem}
\begin{proof}
    By the co-associativity of $\lambda$, we have $\Delta(D) \otimes w = D \otimes \lambda(w)$ for all $w \in A_n$. Therefore,
    $$\Delta(D) \otimes w = D \otimes D \otimes w,$$
    so $\Delta(D) = D \otimes D$. Similarly, the co-unitality of $\lambda$ implies that $\varepsilon(D)w = w$, so $\varepsilon(D) = 1$. This concludes the proof.
\end{proof}

Using the perfect duality between $A_k$ and $A_{n - k}$, we introduce the following notation.

\begin{ntt}\label{ntt:Frobenius duality}
    Fix $w \in A_n \nonzero$. For each $k,\ell \in \{1,\ldots,n\}$, choose bases $\{t_i^{(k)}\}$ of $A_k$ and $\{s_i^{(\ell)}\}$ of $A_\ell$ such that
    \begin{equation}\label{eq:duality of a with b}
        t_i^{(k)} s_j^{(n - k)} = \delta_{ij} w.
    \end{equation}
    Letting $d_k \coloneqq \dim(A_k)$, we define elements $x_{ij}^{(k)}, y_{ij}^{(k)} \in \aut^\ell(A)$ such that
    \begin{equation}\label{eq:lambda on a and b}
        \lambda(t_i^{(k)}) = \sum_{j = 1}^{d_k} x_{ij}^{(k)} \otimes t_j^{(k)}, \quad \lambda(s_i^{(k)}) = \sum_{j = 1}^{d_k} y_{ij}^{(k)} \otimes s_j^{(k)},
    \end{equation}
    and let $X^{(k)} \coloneqq (x_{ij}^{(k)})$ and $Y^{(k)} \coloneqq (y_{ij}^{(k)})$ in $M_{d_k}(A)$.
\end{ntt}

We now describe how the comultiplication and counit of $\aut^\ell(A)$ interact with the elements $x_{ij}^{(k)}$ and $y_{ij}^{(k)}$ defined in Notation \ref{ntt:Frobenius duality}.

\begin{lem}\label{lem:Frobenius Manin coalgebra structure}
    For all $k \in \{1,\dots,n\}$ and $i,j \in \{1,\dots,d_k\}$, we have
    $$\Delta(x_{ij}^{(k)}) = \sum_{\ell = 1}^{d_k} x_{i\ell}^{(k)} \otimes x_{\ell j}^{(k)}, \quad \Delta(y_{ij}^{(k)}) = \sum_{\ell = 1}^{d_k} y_{i\ell}^{(k)} \otimes y_{\ell j}^{(k)},$$
    and $\varepsilon(x_{ij}^{(k)}) = \varepsilon(y_{ij}^{(k)}) = \delta_{ij}$.
\end{lem}
\begin{proof}
    The co-associativity of $\lambda$ implies that
    $$\sum_{j,\ell = 1}^{d_k} x_{ij}^{(k)} \otimes x_{j\ell}^{(k)} \otimes t_\ell^{(k)} = \sum_{j = 1}^{d_k} x_{ij}^{(k)} \otimes \lambda(t_j^{(k)}) = \sum_{j = 1}^{d_k} \Delta(x_{ij}^{(k)}) \otimes t_j^{(k)}.$$
    Swapping $j$ and $\ell$ in the first summation, we get
    $$\sum_{j,\ell = 1}^{d_k} x_{i\ell}^{(k)} \otimes x_{\ell j}^{(k)} \otimes t_j^{(k)} = \sum_{j = 1}^{d_k} \Delta(x_{ij}^{(k)}) \otimes t_j^{(k)},$$
    and thus
    $$\Delta(x_{ij}^{(k)}) = \sum_{\ell = 1}^{d_k} x_{i\ell}^{(k)} \otimes x_{\ell j}^{(k)}$$
    for all $i,j$. By the same argument, we also have $\Delta(y_{ij}^{(k)}) = \sum_{\ell = 1}^{d_k} y_{i\ell}^{(k)} \otimes y_{\ell j}^{(k)}$ for all $i,j$. The counitality of $\lambda$ implies that
    $$\sum_j \varepsilon(x_{ij}^{(k)}) t_j^{(k)} = t_i^{(k)}, \quad \sum_j \varepsilon(y_{ij}^{(k)}) s_j^{(k)} = s_i^{(\ell)},$$
    by \eqref{eq:lambda on a and b}, from which it follows that $\varepsilon(x_{ij}^{(k)}) = \varepsilon(y_{ij}^{(k)}) = \delta_{ij}$.
\end{proof}

Having analyzed the coalgebra structure of $\aut^\ell(A)$, we now move on to the antipode. For a matrix $M = (m_{ij}) \in M_d(\aut^\ell(A))$, define $S(M)$ to be the matrix whose $(i,j)$-entry is $S(m_{ij})$.

\begin{lem}\label{lem:U and V are invertible}
    For all $k \in \{1,\dots,n\}$, the matrices $X^{(k)}$ and $Y^{(k)}$ are invertible, and we have $S(X^{(k)}) = (X^{(k)})^{-1}$ and $S(Y^{(k)}) = (Y^{(k)})^{-1}$.
\end{lem}
\begin{proof}
    We have
    \begin{equation}\label{eq:antipode on u and v}
        \begin{aligned}
            \sum_\ell S(x_{i\ell}^{(k)}) x_{\ell j}^{(k)} &= \sum_\ell x_{i\ell}^{(k)} S(x_{\ell j}^{(k)}) = \varepsilon(x_{ij}^{(k)}) = \delta_{ij}, \\
            \sum_\ell S(y_{i\ell}^{(k)}) y_{\ell j}^{(k)} &= \sum_\ell y_{i\ell}^{(k)} S(y_{\ell j}^{(k)}) = \varepsilon(y_{ij}^{(k)}) = \delta_{ij}.
        \end{aligned}
    \end{equation}
    Then \eqref{eq:antipode on u and v} gives
    $$S(X^{(k)})X^{(k)} = X^{(k)}S(X^{(k)}) = I_{d_k}, \quad S(Y^{(k)})Y^{(k)} = Y^{(k)}S(Y^{(k)}) = I_{d_k},$$
    in other words, $X^{(k)}$ and $Y^{(k)}$ are invertible, and we have $S(X^{(k)}) = (X^{(k)})^{-1}$ and $S(Y^{(k)}) = (Y^{(k)})^{-1}$.
\end{proof}

As we show next, the matrices $X^{(k)}$ and $Y^{(n - k)}$ are inverse to each other, up to multiplication by $D^{-1}$ and transposition.

\begin{prop}\label{prop:Y is inverse of X}
    For all $k \in \{1,\dots,n\}$, we have $X^{(k)}(Y^{(n - k)})^{\mathsf{T}} = D I_n$. Consequently, $(X^{(k)})^{-1} = (Y^{(n - k)})^{\mathsf{T}} D^{-1}$, and
    $$S(x_{ij}^{(k)}) = y_{ji}^{(n - k)}D^{-1}$$
    for all $i,j \in \{1,\dots,d_k\}$.
\end{prop}
\begin{proof}
    It follows from \eqref{eq:duality of a with b} that
    \begin{equation}\label{eq:lambda dual bases part 1}
        \lambda(t_i^{(k)}s_j^{(n - k)}) = \delta_{ij} \lambda(w) = \delta_{ij} D \otimes w.
    \end{equation}
    On the other hand, $\lambda$ is an algebra homomorphism, so we get
    \begin{equation}\label{eq:lambda dual bases part 2}
        \begin{aligned}
            \lambda(t_i^{(k)}s_j^{(n - k)}) &= \lambda(t_i^{(k)}) \lambda(s_j^{(n - k)}) = \sum_{a,b = 1}^{d_k} x_{ia}^{(k)} y_{jb}^{(n - k)} \otimes t_a^{(k)}s_b^{(n - k)} \\
            &= \sum_{a,b = 1}^{d_k} \delta_{ab} x_{ia}^{(k)} y_{jb}^{(n - k)} \otimes w \quad (\text{by \eqref{eq:duality of a with b}}) \\
            &= \sum_{\ell = 1}^{d_k} x_{i\ell}^{(k)} y_{j\ell}^{(n - k)} \otimes w.
        \end{aligned}
    \end{equation}
    Combining \eqref{eq:lambda dual bases part 1} and \eqref{eq:lambda dual bases part 2}, we deduce that $\sum_{\ell = 1}^{d_k} x_{i\ell}^{(k)} y_{j\ell}^{(n - k)} = \delta_{ij} D$. Therefore,
    $$X^{(k)}(Y^{(n - k)})^{\mathsf{T}} = D I_{d_k}$$
    for all $k \in \{1,\ldots,n\}$. By Lemma \ref{lem:U and V are invertible}, it follows that $S(X^{(k)}) = (X^{(k)})^{-1} = (Y^{(n - k)})^{\mathsf{T}} D^{-1}$. We conclude that $S(x_{ij}^{(k)}) = y_{ji}^{(n - k)}D^{-1}$, which finishes the proof.
\end{proof}

We summarize our findings in the next result.

\begin{prop}\label{prop:Frobenius Manin Hopf algebra}
    The Hopf algebra $\aut^\ell(A)$ contains elements $x_{ij} = x_{ij}^{(1)}$ for $i,j \in \{1,\dots,n\}$ and $D^{\pm 1}$, subject to the relations which make $\lambda$ into an algebra homomorphism, and the relations
    $$X^{(k)}(Y^{(n - k)})^{\mathsf{T}} D^{-1} = (Y^{(n - k)})^{\mathsf{T}} D^{-1}X^{(k)} = I_{d_k},$$
    for all $k \in \{1,\ldots,n\}$, where we use the notation from Notation \ref{ntt:Frobenius duality}.

    The comultiplication and counit on the elements $x_{ij}$ is given by
    $$\Delta(x_{ij}) = \sum_{k = 1}^n x_{ik} \otimes x_{kj}, \quad \varepsilon(x_{ij}) = \delta_{ij}.$$
    The antipode on the elements $x_{ij}$ is given by $S(x_{ij}) = y_{ji}^{(n - 1)} D^{-1}$.
\end{prop}
\begin{proof}
    By the discussion above, the relations and the formulas for the comultiplication, counit, and antipode given in the statement are all necessary to have a well-defined left coaction on $A$.
\end{proof}

Note that we do not claim that the relations in Proposition \ref{prop:Frobenius Manin Hopf algebra} are sufficient to define a Hopf algebra, but they are certainly necessary. In general, more generators and relations may be required, which can be obtained via Manin's \emph{Hopf envelope} construction \cite[Chapter 8]{ManinBook}.
We work this out in detail in the next section, but only for the special case of the algebras $A_q^!$.

\section{Construction of \texorpdfstring{$\aut^r(A_q)$}{aut(Aq)}}\label{sec:Manin Hopf algebra of Aq}

As mentioned before, we can construct $\aut^r(A_q)$ by describing $\aut^\ell(A_q^!)$ using the discussion from Section \ref{sec:Frobenius}. This is because $\aut^r(A_q) \cong \aut^\ell(A_q^!)$, by the following well-known result.

\begin{prop}[{\cite[Theorem 6.10]{ManinBook}, \cite[Lemma 2.1.5(3)]{Hexagon}}]\label{prop:left algebra of dual}
    Let $A$ be a quadratic algebra. Then $\aut^r(A) \cong \aut^\ell(A^!)$.
\end{prop}

We now proceed to construct the Hopf algebra $\aut^\ell(A_q^!) \cong \aut^r(A_q)$. This Hopf algebra has been studied before, but it is difficult to find the explicit presentation of $\aut^r(A_q(n))$ we want in the literature, in particular for arbitrary $q \in \kk^*$ and arbitrary number of variables $n$. For example, \cite[Appendix A]{RaedscheldersVanDenBergh} describes the Hopf algebra $\aut^r(A_1(n))$ for arbitrary $n$, while \cite[Lemma 5.1]{ChanWaltonZhang} describes $\aut^r(A_q(2))$ for arbitrary $q \in \kk^*$.  See also \cite[Example 9.6]{ManinBook}, \cite{RaedscheldersVanDenBergh2}, and \cite{ChirvasituWaltonWang}.

We keep the notation from Notation \ref{ntt:Frobenius duality} with
$$A = A_q^! = \frac{\kk\ang{t_1,\ldots,t_n}}{(t_i^2, t_j t_i + q^{-1} t_i t_j \mid j > i)},$$
where we make the following choices.

\begin{ntt}\label{ntt:I_k indexing set}
    Define $w \coloneqq t_1 \dots t_n$. For $k \in \{1,\dots,n\}$, define 
    $$\II_k \coloneqq \{\*r = (r_1,\dots,r_k) \in \ZZ^k \mid 1 \leq r_1 < r_2 < \dots < r_k \leq n\}.$$
    Note that $\II_1 = \{1,\dots,n\}$.  
    Given $\*r \in \II_k$, we define $t_{\*r} \coloneqq t_{r_1} t_{r_2} \dots t_{r_k}$, so that $\{ t_{\*r} \}$ is a $\kk$-basis of $(A^!_q)_k$ indexed by $\*r \in \II_k$. 

    For $\*r \in \II_k$, let $\widehat{\*r}$ be the element of $\II_{n - k}$ obtained by removing $r_1,\dots,r_k$ from $(1,\dots,n)$, and let $\sigma_{\*r} \in S_n$ be the permutation
    $$\sigma_{\*r} \coloneqq \begin{pmatrix}
        1 & \dots & k & k + 1 & \dots & n \\
        r_1 & \dots & r_k & \widehat{r}_1 & \dots & \widehat{r}_{n - k}
    \end{pmatrix}.$$
\end{ntt}

For all $\*r \in \II_k$, we have
$$\lambda(t_{\*r}) = \sum_{\*c \in \II_k} x_{\*r\*c} \otimes t_{\*c}.$$
for some elements $x_{\*r\*c} \in  \aut^\ell(A)$. Note that we have dropped the superscript $(k)$ from $t_{\*c}$ and $x_{\*r\*c}$ in this equation, since it is implied by the fact that $\*r, \*c \in \II_k$. We will continue to drop the superscript below for variables indexed by $\II_k$, unless it is needed for clarity.

We introduce quantum determinants and quantum minors, which will be important for our presentation of $\aut^r(A_q)$.

\begin{dfn}\label{dfn:qdet and inversion number}
    Let $S$ and $T$ be subsets of $\{1, \dots, n\}$ for some $n$. Given a function $f \colon S \to T$, we define its \emph{inversion set} as
    $$\Inv(f) \coloneqq \{(i,j) \in S \times S \mid i < j \text{ and } f(i) > f(j)\}.$$
    The \emph{inversion number} of $f$ is $I(f) \coloneqq |\Inv(f)|$. Given a $k \times k$ matrix $P = (p_{ij}) \in M_n(R)$ over some $\kk$-algebra $R$, we define the \emph{quantum determinant} of $P$ as follows:
    $$\qdet(P) \coloneqq \sum_{\sigma \in S_n} (-q)^{-I(\sigma)} p_{1\sigma(1)} \dots p_{n\sigma(n)}.$$
    Define $X \coloneqq (x_{ij}) \in M_n(\aut^r(A_q))$. Of course, $X = X^{(1)}$ from Notation \ref{ntt:Frobenius duality}. Given $\*r, \*c \in \II_k$, we define $X_{\*r \*c}$ to be the $k \times k$ sub-matrix of $X$ using rows $r_1,\dots,r_k$ and columns $c_1,\dots,c_k$. The quantum determinant $\qdet(X_{\*r \*c})$ is known as a \emph{quantum minor} of $X$.
\end{dfn}

For example,
$$\qdet\begin{pmatrix}
    a & b \\
    c & d
\end{pmatrix} = ad - q^{-1} bc.$$
We also called the element $D$ from the previous section a quantum determinant. The next result shows that there is no conflict of terminology.

\begin{lem}\label{lem:Manin qdet}
    We have $D = \qdet(X)$, where $X = (x_{ij}) \in M_n(\aut^r(A_q))$.
\end{lem}
\begin{proof}
    Recall that $D$ is defined by the property that $\lambda(w) = D \otimes w$, where $w = t_1 \dots t_n$. We have
    \begin{align*}
        \lambda(w) &= \lambda(t_1 \dots t_n) = \lambda(t_1) \dots \lambda(t_n) \\
        &= \left(\sum_{j = 1}^n x_{1j} \otimes t_j\right) \dots \left(\sum_{j = 1}^n x_{nj} \otimes t_j\right) \\
        &= \sum_{\sigma \in S_n} x_{1\sigma(1)} \dots x_{n\sigma(n)} \otimes t_{\sigma(1)} \dots t_{\sigma(n)} \\
        &= \sum_{\sigma \in S_n} (-q)^{-I(\sigma)} x_{1\sigma(1)} \dots x_{n\sigma(n)} \otimes t_1 \dots t_n \\
        &= \qdet(X) \otimes w,
    \end{align*}
    and therefore $D = \qdet(X)$.
\end{proof}

Next, we state the main result of this section, which fully describes the universal right coacting Hopf algebra $\aut^r(A_q)$.

\begin{thm}\label{thm:description of universal Hopf algebra}
    The Hopf algebra $\aut^r(A_q)$ is generated by $x_{ij}$ for $1 \leq i,j \leq n$ and $D^{-1}$, subject to the following relations:
    \begin{subequations}\label{eq:Manin full relations}
        \begin{gather}
            x_{ik}x_{ij} = q x_{ij}x_{ik} \quad (k > j),\label{eq:Manin q commute} \\
            ad - q^{-1}bc = \qdet\begin{pmatrix}
                a & b \\
                c & d
            \end{pmatrix} = da - q cb,\label{eq:Manin 2x2 minor} \\
            \sum_{k = 1}^n (-q)^{k - i} \qdet(X_{\hat k \hat\imath}) D^{-1} x_{k j} = \delta_{ij},\label{eq:Manin long relation} \\
            D^{-1} D = D D^{-1} = 1,
        \end{gather}
    \end{subequations}
    for all $i,j,k \in \{1,\dots,n\}$, and all $2 \times 2$ sub-matrices $\begin{pmatrix}
        a & b \\
        c & d
    \end{pmatrix}$ of $X$, where $D \coloneqq \qdet(X)$ is the quantum determinant of the matrix $X = (x_{ij})$. The Hopf structure of $\aut^r(A_q)$ is given by
    \begin{align}
        \Delta(x_{ij}) &= \sum_{k = 1}^n x_{ik} \otimes x_{kj}, \nonumber \\
        \varepsilon(x_{ij}) &= \delta_{ij}, \label{eq:Manin Hopf structure}\\
        S(x_{ij}) &= (-q)^{j - i} \qdet(X_{\hat{\jmath}\hat{\imath}})D^{-1}. \nonumber
    \end{align}
\end{thm}

\begin{rem}
    One way to prove Theorem \ref{thm:description of universal Hopf algebra} would be to take the presentation of $\aut^r(A_1)$ from \cite[Appendix A]{RaedscheldersVanDenBergh} and twist it by an appropriate \emph{$2$-cocycle}, as in \cite[Theorem 2.2.3]{Hexagon}. We opt for a more direct approach, as the presentation given in \cite{RaedscheldersVanDenBergh} relies on a lot of technical machinery. Furthermore, our approach has the advantage that \eqref{eq:Manin full relations} excludes some unnecessary relations given in \cite[Appendix A]{RaedscheldersVanDenBergh}. In particular, Raedschelders and Van den Bergh give relations of the form
    $$\sum_{\*m \in \II_k} (-q)^{I(\sigma_{\*m}) - I(\sigma_{\*r})} \qdet(X_{\widehat{\*m} \widehat{\*r}}) D^{-1} \qdet(X_{\*m \*c}) = \delta_{\*r \*c},$$
    which are in fact consequences of \eqref{eq:Manin full relations}. This is explained in more detail in Remark \ref{rem:redundant relations}.
\end{rem}

 We will use results from \cite{ChervovFalquiRubtsovSilantyev} to help prove Theorem~\ref{thm:description of universal Hopf algebra}, so we recall some of their terminology next.

\begin{dfn}
    Let $A$ be a $\kk$-algebra, and let $q \in \kk^*$. A \emph{$q$-Manin matrix in $A$} is a matrix $P = (p_{ij}) \in M_n(A)$ whose entries satisfy $p_{ik} p_{ij} = q p_{ij} p_{ik}$ for all $i,j,k \in \{1,\dots,n\}$ with $j < k$, and
    $$ad - q^{-1} bc = da - q cb$$
    for all $2 \times 2$ sub-matrices $\begin{pmatrix}
        a & b \\
        c & d
    \end{pmatrix}$ of $P$.  In other words, the entries of $P$ satisfy the relations \eqref{eq:Manin q commute} and \eqref{eq:Manin 2x2 minor} upon substituting $x_{ij} \mapsto p_{ij}$.
\end{dfn}

\begin{rem}
    The definition of a $q$-Manin matrix in \cite{ChervovFalquiRubtsovSilantyev} is transposed from our convention. In other words, the authors of \cite{ChervovFalquiRubtsovSilantyev} require the entries in the same column to $q$-commute, while we require the entries in the same row to $q$-commute. Similarly, the relation coming from $2 \times 2$ quantum minors of $P$ is also transposed. We will apply several results from \cite{ChervovFalquiRubtsovSilantyev} by implicitly translating their results into our conventions without further comment.
\end{rem}

The following result from \cite{ChervovFalquiRubtsovSilantyev} extends some basic facts about determinants (namely, expansion by minors, and the adjoint formula for the inverse) to the $q$-matrix case.

\begin{prop}[{\cite[Proposition 3.3(6), and Theorems 4.5 and 4.7]{ChervovFalquiRubtsovSilantyev}}]\label{prop:q-Manin matrix}
    Let $R$ be a $\kk$-algebra, let $q \in \kk^*$, and let $P \in M_n(R)$ be a $q$-Manin matrix. Then the following hold.
    \begin{enumerate}
        \item We have
        $$\sum_{\*m \in \II_k} (-q)^{I(\sigma_{\*c}) - I(\sigma_{\*m})} \qdet(P_{\*r \*m}) \qdet(P_{\widehat{\*c} \widehat{\*m}}) = \delta_{\*r \*c} \qdet(P)$$
        for all $k \in \{1,\dots,n - 1\}$ and $\*r, \*c \in \II_k$.\label{item:redundant relations}
        \item If $P$ and $\qdet(P)$ are invertible, then $P^{-1}$ is a $q^{-1}$-Manin matrix and $\invqdet(P^{-1}) = \qdet(P)^{-1}$. Furthermore,
        $$\invqdet((P^{-1})_{\*r\*c}) = (-q)^{I(\sigma_{\*c}) - I(\sigma_{\*r})} \qdet(P_{\widehat{\*c} \widehat{\*r}}) \qdet(P)^{-1}$$
        for all $k \in \{1,\dots,n - 1\}$ and $\*r,\*c \in \II_k$.\label{item:invertible q-Manin matrix}
    \end{enumerate}
\end{prop}

To prove Theorem \ref{thm:description of universal Hopf algebra}, we will follow the method of Manin in \cite{ManinBook}, where the universal coacting bialgebra $\ndo^{\ell}(A)$ is constructed first, and then $\aut^{\ell}(A)$ is constructed from it as a  \emph{Hopf envelope}.

\begin{lem}
    \label{lem:presentation of End}
    The bialgebra $\ndo^{\ell}(A)$ is generated as an algebra by $x_{ij}$ for $1 \leq i, j \leq n$, subject to the relations \eqref{eq:Manin q commute} and \eqref{eq:Manin 2x2 minor}.  In particular, every minor of $X = (x_{ij})$ is a $q$-Manin matrix.  The coalgebra structure is given by $\Delta, \varepsilon$ as in \eqref{eq:Manin Hopf structure}, and the coaction $\lambda \colon A \to \ndo^{\ell}(A) \otimes A$ is given by $t_i \mapsto \sum_j x_{ij} \otimes t_j$. 
\end{lem}
\begin{proof}
    Recall that 
    $$A = A_q^! = \frac{\kk\ang{t_1,\ldots,t_n}}{(t_i^2, t_j t_i + q^{-1} t_i t_j \mid j > i)},$$
    and therefore $$A^! = (A_q)^{!!} = A_q = \frac{\kk\ang{x_1, \dots, x_n}}{(x_{\ell} x_k - q x_k x_{\ell} \mid \ell> k)}.$$
    According to \cite[Equation (6.4)]{ManinBook}, $\ndo^{\ell}(A)$ is generated by elements $\{ x_{ij} \mid 1 \leq i, j \leq n \}$, with one relation $r_{\alpha \beta} = \sum_{i,j,k,\ell} c_{ij}^{\alpha} d_{k\ell}^{\beta} x_{ik} x_{j\ell}$ for each pair of relations $r_{\alpha} = \sum_{ij} c_{ij}^{\alpha} t_i t_j$ of $A$ and $r_{\beta} = \sum_{k\ell} d_{\beta}^{k\ell} x_k x_{\ell}$ of $A^!$.  

    Taking $r_{\alpha} = t_i^2$ and $r_{\beta} = x_{\ell} x_k - q x_k x_\ell$ gives a relation $r_{\alpha \beta} = x_{i\ell}x_{ik} - q x_{ik}x_{i\ell}$, for all $i$ and all $\ell > k$.  These are exactly the relations \eqref{eq:Manin q commute}.    Taking $r_{\alpha} = t_j t_i + q^{-1} t_i t_j$ instead gives a relation $r_{\alpha \beta} = x_{j \ell} x_{ik} - q x_{jk}x_{i\ell} - x_{ik}x_{j\ell} + q^{-1} x_{i\ell}x_{jk}$ for all $j > i$ and $\ell > k$, which is relation \eqref{eq:Manin 2x2 minor} for the $2 \times 2$ minor $\begin{pmatrix} x_{ik} & x_{i\ell} \\x_{jk} & x_{j\ell} \end{pmatrix}$ of $X = (x_{ij})$.  Thus $\ndo^{\ell}(A)$ is presented by $\{ x_{ij}\}$ with the relations \eqref{eq:Manin q commute} and \eqref{eq:Manin 2x2 minor} as claimed.

    In \cite[Sections 6.7 and 6.8]{ManinBook} it is shown that $\ndo^{\ell}(A)$ is indeed a bialgebra when given the coalgebra structure with the formulas in \eqref{eq:Manin Hopf structure}. Moreover, $\lambda \colon A \to \ndo^{\ell}(A) \otimes A$ with the given formula is the universal graded coaction by a bialgebra making $A$ into a left comodule algebra \cite[Sections 6.4--6.6]{ManinBook}.
\end{proof}

We are now ready to prove Theorem \ref{thm:description of universal Hopf algebra}.

\begin{proof}[Proof of Theorem \ref{thm:description of universal Hopf algebra}]
    As above, write 
    $$A = A_q^! = \frac{\kk\ang{t_1,\ldots,t_n}}{(t_i^2, t_j t_i + q^{-1} t_i t_j \mid j > i)}.$$
    By Lemma~\ref{lem:presentation of End}, $\ndo^{\ell}(A) = \kk\langle z_{ij} \rangle/(R_0)$, where we have chosen to use the variables $z_{ij}$ instead of $x_{ij}$ to avoid confusion below.  The relations $R_0$ are exactly the ones in \eqref{eq:Manin q commute} and \eqref{eq:Manin 2x2 minor} (written in terms of $z$), with coalgebra structure given by $\Delta, \varepsilon$ as in \eqref{eq:Manin Hopf structure}, and the coaction $\lambda \colon A \to \ndo^{\ell}(A) \otimes A$ given by $t_i \mapsto \sum_j z_{ij} \otimes t_j$.  
    
    Manin's Hopf envelope construction from \cite[Chapter 8]{ManinBook} can be used to obtain $\aut^{\ell}(A)$ from $\ndo^{\ell}(A)$.  In detail, by \cite[Theorem 8.3]{ManinBook}, the Hopf algebra $\aut^{\ell}(A)$ is generated by the entries of infinitely many $n \times n$ matrices $\wt{Z}_0, \wt{Z}_1, \wt{Z}_2, \dots$ subject to the following relations:
    \begin{enumerate}
        \item The elements of $R_0$ written for $\wt{Z}_k$ if $k$ is even.
        \item The elements of $R_0^{\mathrm{op}}$ written for $\wt{Z}_k$ if $k$ is odd.
        \item $\wt{Z}_k \wt{Z}_{k + 1} = \wt{Z}_{k + 1} \wt{Z}_{k} = I_n$ if $k$ is even.
        \item $\wt{Z}_k^\mathsf{T} \wt{Z}_{k + 1}^\mathsf{T} = \wt{Z}_{k + 1}^\mathsf{T} \wt{Z}_{k}^\mathsf{T} = I_n$ if $k$ is odd.
    \end{enumerate}
    The Hopf structure of this Hopf algebra is given by
    $$\Delta(\wt{Z}_k) = \begin{cases}
        \wt{Z}_k \otimes \wt{Z}_k, &\text{if } k \text{ is even}, \\
        (\wt{Z}_k^\mathsf{T} \otimes \wt{Z}_k^\mathsf{T})^\mathsf{T}, &\text{if } k \text{ is odd},
    \end{cases} \qquad \varepsilon(\wt{Z}_k) = I_n, \qquad S(\wt{Z}_k) = \wt{Z}_{k + 1},$$
    where the maps are applied entry-wise, and the tensor product $P \otimes Q$ of two $n \times n$ matrices $P = (p_{ij})$ and $Q = (q_{ij})$ is defined by
    $$(P \otimes Q)_{ij} \coloneqq \sum_{k = 1}^n p_{ik} \otimes q_{kj}.$$
    By construction, the map $\ndo^{\ell}(A) \to \aut^{\ell}(A)$ given by $z_{ij} \mapsto (\wt{Z}_0)_{ij}$ is a bialgebra map and it is universal for maps from $\ndo^{\ell}(A)$ to Hopf algebras by \cite[Theorem 8.3]{ManinBook}.  Together with the universal property of $\ndo^{\ell}(A)$ in \cite[Lemma 6.6]{ManinBook} this shows 
    that $\aut^{\ell}(A)$ is the universal Hopf algebra left coacting on $A$, with the same formula for the coaction.
    
    Let $H$ be the algebra generated by elements $x_{ij}$ and 
    $D^{-1}$, with relations \eqref{eq:Manin full relations}.  We want to show that $H \cong \aut^{\ell}(A)$.
    Let $X = (x_{ij})$.
    Define the matrices $Z_k \in M_n(H)$ by
    $$(Z_k)_{ij} \coloneqq \begin{cases}
        (-q)^{k(j - i)} D^{\frac{k}{2}} x_{ij} D^{-\frac{k}{2}}, &\text{if } k \text{ is even}, \\
        (-q)^{k(j - i)} D^{\frac{k - 1}{2}} \qdet(X_{\hat\jmath \hat\imath}) D^{-\frac{k + 1}{2}}, &\text{if } k \text{ is odd}.  
    \end{cases}$$
    Throughout the rest of the proof, if $\pi$ is a homomorphism and $M, M'$ are matrices we will write $\pi(M) = M'$
    to mean $\pi(M_{ij}) = M'_{ij}$ for all $i,j$.
    We claim that there is an algebra map $\pi \colon \aut^{\ell}(A) \to H$ such that $\pi(\wt{Z}_k) = Z_k$ for all $k$.

    First, we check that $Z_k Z_{k + 1} = Z_{k + 1} Z_{k} = I_n$ if $k$ is even. Indeed, we have
    $$(Z_k Z_{k + 1})_{ij} = (-q)^{k(j - i)} D^{\frac{k}{2}} \left(\sum_{\ell = 1}^n (-q)^{j - \ell} x_{i\ell} \qdet(X_{\hat\jmath \hat\ell})\right) D^{-\left(\frac{k}{2} + 1\right)}.$$
    By definition $X$ satisfies the relations \eqref{eq:Manin q commute} and \eqref{eq:Manin 2x2 minor}, in other words $X$ is a $q$-Manin matrix.   Then 
    $$\sum_{\ell = 1}^n (-q)^{j - \ell} x_{i\ell} \qdet(X_{\hat\jmath \hat\ell}) = \delta_{ij} D,$$
    by Proposition \ref{prop:q-Manin matrix}\eqref{item:redundant relations}, and thus $(Z_k Z_{k + 1})_{ij} = \delta_{ij}$. For the reverse order, consider
    $$(Z_{k + 1} Z_k)_{ij} = (-q)^{k(j - i)} D^{\frac{k}{2}} \left(\sum_{\ell = 1}^n (-q)^{\ell - i} \qdet(X_{\hat\ell \hat\imath}) D^{-1} x_{\ell j}\right) D^{-\frac{k}{2}}.$$
    In this case, we can use \eqref{eq:Manin long relation} to get
    $$\sum_{\ell = 1}^n (-q)^{\ell - i} \qdet(X_{\hat\ell \hat\imath}) D^{-1} x_{\ell j} = \delta_{ij},$$
    from which it follows that $(Z_{k + 1} Z_k)_{ij} = \delta_{ij}$. We conclude that $Z_k Z_{k + 1} = Z_{k + 1} Z_k = I_n$ for all even integers $k$, as required.

    Next, we check that $Z_k^\mathsf{T} Z_{k + 1}^\mathsf{T} = Z_{k + 1}^\mathsf{T} Z_{k}^\mathsf{T} = I_n$ if $k$ is odd. We have
    $$(Z_k^\mathsf{T} Z_{k + 1}^\mathsf{T})_{ij} = (-q)^{k(i - j)} D^{\frac{k - 1}{2}} \left(\sum_{\ell = 1}^n (-q)^{\ell - j} \qdet(X_{\hat\imath \hat\ell}) x_{j\ell}\right) D^{-\frac{k + 1}{2}} = \delta_{ij},$$
    by Proposition \ref{prop:q-Manin matrix}\eqref{item:redundant relations}. Similarly, multiplying the matrices in the reverse order gives
    \begin{equation}\label{eq:transpose Z_k relation}
        (Z_{k + 1}^\mathsf{T} Z_k^\mathsf{T})_{ij} = (-q)^{k(i - j)} D^{\frac{k + 1}{2}} \left(\sum_{\ell = 1}^n (-q)^{i - \ell} x_{\ell i} D^{-1} \qdet(X_{\hat\ell \hat\jmath})\right) D^{-\frac{k + 1}{2}}.
    \end{equation}
    We claim that
    \begin{equation}\label{eq:redundant long relation}
        \sum_{\ell = 1}^n (-q)^{i - \ell} x_{\ell i} D^{-1} \qdet(X_{\hat\ell \hat\jmath}) = \delta_{ij}.
    \end{equation}
    As shown above, we have $Z_1 = Z_0^{-1} = X^{-1}$. Since $X$ is a $q$-Manin matrix by construction, it follows by Proposition \ref{prop:q-Manin matrix}\eqref{item:invertible q-Manin matrix} that $Z_1$ is a $q^{-1}$-Manin matrix. By Proposition \ref{prop:q-Manin matrix}\eqref{item:redundant relations}, we therefore have
    \begin{equation}\label{eq:qdet of Z1}
        \sum_{\ell = 1}^n (-q)^{j - \ell} \invqdet((Z_1)_{\hat\imath \hat\ell}) (Z_1)_{j\ell} = \delta_{ij} \invqdet(Z_1) = \delta_{ij} D^{-1},
    \end{equation}
    since $\invqdet(Z_1) = \qdet(X)^{-1} = D^{-1}$ by Proposition \ref{prop:q-Manin matrix}\eqref{item:invertible q-Manin matrix}. Again applying Proposition \ref{prop:q-Manin matrix}\eqref{item:invertible q-Manin matrix}, we have
    $$\invqdet((Z_1)_{\hat\imath \hat\ell}) = (-q)^{i - \ell} x_{\ell i} D^{-1}.$$
    Therefore, \eqref{eq:qdet of Z1} becomes
    $$\sum_{\ell = 1}^n (-q)^{i - \ell} x_{\ell i} D^{-1} \qdet(X_{\hat\ell \hat\jmath}) D^{-1} = \delta_{ij} D^{-1},$$
    where we used that $(Z_1)_{j\ell} = (-q)^{\ell - j} \qdet(X_{\hat\ell \hat\jmath}) D^{-1}$ by definition. This proves the claim.
    
    Substituting \eqref{eq:redundant long relation} into \eqref{eq:transpose Z_k relation}, we get
    $$(Z_{k + 1}^\mathsf{T} Z_k^\mathsf{T})_{ij} = (-q)^{k(i - j)} D^{\frac{k + 1}{2}} \left(\sum_{\ell = 1}^n (-q)^{i - \ell} x_{\ell i} D^{-1} \qdet(X_{\hat\ell \hat\jmath})\right) D^{-\frac{k + 1}{2}} = \delta_{ij}.$$
    It follows that $Z_k^\mathsf{T} Z_{k + 1}^\mathsf{T} = Z_{k + 1}^\mathsf{T} Z_k^\mathsf{T} = I_n$ for all odd integers $k$, as required.

    It remains to check that $Z_k$ satisfies $R_0$ if $k$ is even and $R_0^\mathrm{op}$ if $k$ is odd. We know that $Z_0 = X$ satisfies $R_0$ by definition. Furthermore, it is easy to see from the definition of $Z_k$ that $Z_k$ also satisfies $R_0$ for all even $k$. In other words, $Z_k$ is a $q$-Manin matrix for all even $k$. By Proposition \ref{prop:q-Manin matrix}\eqref{item:invertible q-Manin matrix}, it follows that $Z_k^{-1}$ is a $q^{-1}$-Manin matrix if $k$ is even. Since $Z_k^{-1} = Z_{k + 1}$, this means that $Z_{k + 1}$ is a $q^{-1}$-Manin matrix.  But it is easy to check that being a $q^{-1}$-Manin matrix is equivalent to the matrix entries satisfying the relations in  $R_0^\mathrm{op}$.  We have now checked all of the relations, and so there is indeed a well-defined algebra map $\pi \colon \aut^{\ell}(A) \to H$ with the formula $\wt{Z}_k \mapsto Z_k$.  
    
    Because $\aut^{\ell}(A)$ coacts on $A$ via $\lambda(t_i) = \sum_{j} (\wt{Z}_0)_{ij} \otimes t_j$, the results of Section~\ref{sec:Frobenius} apply, with the matrix $\wt{Z}_0$ in the role of $X$.  There is a quantum determinant 
    $\wt{D} \in \aut^{\ell}(A)$ such that $\lambda(w) = \wt{D} \otimes w$, and there is a matrix $Y = (y_{ij}) \in M_n(\aut^{\ell}(A))$ such that $\wt{Z}_0 Y^{\mathsf{T}} = \wt{D} I_n$ by Proposition~\ref{prop:Y is inverse of X}. Thus, $\wt{Z}_1 = \wt{Z}_0^{-1} = Y^{\mathsf{T}} \wt{D}^{-1}$. By Lemma~\ref{lem:Manin qdet}, $\wt{D} = \qdet(\wt{Z}_0)$.  Since $\wt{D} = \det_q(\wt{Z}_0)$ is a unit in $\aut^{\ell}(A)$, 
    Proposition~\ref{prop:q-Manin matrix}\eqref{item:invertible q-Manin matrix} applied with $\*r = \{i \}$ and $\*c = \{ j \}$  implies that
    $(\wt{Z}_0^{-1})_{ij} = (-q)^{j-i} \det_q((\wt{Z}_0)_{\hat{\jmath} \hat{\imath}}) D^{-1}$.  Comparing the two formulas for $\wt{Z}_0^{-1}$ we see that $Y_{ij} = (-q)^{i-j} \det_q((\wt{Z}_0)_{\hat{\imath} \hat{\jmath}})$ for all $i,j$.
    The equation $\wt{Z}_1 \wt{Z}_0 = I$ can be written as $Y^\mathsf{T} \wt{D}^{-1} \wt{Z}_0 = I$, or 
    $$\sum_{k = 1}^n (-q)^{k - i} \qdet((\wt{Z}_0)_{\hat k \hat\imath}) \wt{D}^{-1} (\wt{Z}_0)_{k j} = \delta_{ij}.$$
    We have now checked the analogs of all of the relations in \eqref{eq:Manin full relations} for $\wt{Z}_0$ and $\wt{D}^{-1}$, so we conclude 
    that there is a homomorphism $\psi \colon H \to \aut^{\ell}(A)$ with $\psi(X) = \wt{Z}_0$ and $\psi(D^{-1}) = \wt{D}^{-1}$.

    We now check that $\pi$ and $\psi$ are inverse to each other. It is obvious that $\pi \circ \psi = \id_H$. We claim that $\psi(Z_i) = \wt{Z}_i$ for all $i \geq 0$, which we prove by induction.   
    The base case $\psi(Z_0) = \wt{Z}_0$ is clear.  If $\psi(Z_i) = \wt{Z}_i$ holds for even $i$, then the relations $\wt{Z}_{i+1} = \wt{Z}_i^{-1}$ and $Z_{i+1} = Z_i^{-1}$ imply that
    $$\psi(Z_{i+1}) = \psi(Z_i^{-1}) = \psi(Z_i)^{-1} = \wt{Z}_i^{-1} = \wt{Z}_{i+1},$$
    meaning $\psi(Z_{i + 1}) = \wt{Z}_{i + 1}$ is forced. Similarly if $\psi(Z_i) = \wt{Z}_i$ holds for an odd $i$, then $\psi(Z_i^\mathsf{T}) = \wt{Z}_i^\mathsf{T}$ also, and the relations
    $\wt{Z}_{i+1}^\mathsf{T} = (\wt{Z}_i^{\mathsf{T}})^{-1}$ and $Z_{i+1}^\mathsf{T} = (Z_i^{\mathsf{T}})^{-1}$ force 
    $\psi(Z_{i+1}^\mathsf{T}) = \wt{Z}_{i+1}^\mathsf{T}$. It follows that $\psi(Z_{i+1}) = \wt{Z}_{i+1}$, which concludes the induction.
    
    With this claim in hand it is now clear that $\psi \circ \pi = \id_{\aut^{\ell}(A)}$. Therefore, we have proved that $H \cong \aut^\ell(A) \cong \aut^r(A_q)$, as required.
\end{proof}

\begin{rem}\label{rem:redundant relations}
    Notice that some of the relations from Proposition \ref{prop:Frobenius Manin Hopf algebra} are not present in \eqref{eq:Manin full relations}. Indeed, not all of the relations coming from the matrix equations
    $$X^{(k)} (Y^{(n - k)})^{\mathsf{T}} D^{-1} = (Y^{(n - k)})^{\mathsf{T}} D^{-1} X^{(k)} = I$$
    appear in \eqref{eq:Manin full relations}. By Proposition \ref{prop:q-Manin matrix}\eqref{item:redundant relations}, the relations $X^{(k)} (Y^{(n - k)})^{\mathsf{T}} D^{-1} = I$ are consequences of \eqref{eq:Manin q commute} and \eqref{eq:Manin 2x2 minor} and the definition of $D$, so they are not necessary for the presentation of $\aut^r(A_q)$.

    To see why the relations $(Y^{(n - k)})^{\mathsf{T}} D^{-1} X^{(k)} = I$ for $k \neq 1$ are also not necessary, we can apply results from \cite{ChervovFalquiRubtsovSilantyev}. Let $W \coloneqq Z_1$ in the notation from the proof of Theorem \ref{thm:description of universal Hopf algebra}. In other words,
    $$W_{ij} = (-q)^{j - i} \qdet(X_{\hat\jmath \hat\imath}) D^{-1}.$$
    We know that $W = X^{-1}$, and therefore $W$ is a $q^{-1}$-Manin matrix, by Proposition \ref{prop:q-Manin matrix}\eqref{item:invertible q-Manin matrix}. It follows from Proposition \ref{prop:q-Manin matrix}\eqref{item:redundant relations} that
    \begin{equation}\label{eq:W qdet}
        \sum_{\*m \in \II_k} (-q)^{I(\sigma_{\*m}) - I(\sigma_{\*c})} \invqdet(W_{\*r \*m}) \invqdet(W_{\widehat{\*c} \widehat{\*m}}) = \delta_{\*r \*c} \invqdet(W) = \delta_{\*r \*c} D^{-1},
    \end{equation}
    since $\invqdet(W) = D^{-1}$ by Proposition \ref{prop:q-Manin matrix}\eqref{item:invertible q-Manin matrix}. By \cite[Theorem 4.5]{ChervovFalquiRubtsovSilantyev}, we have
    $$\invqdet(W_{\*r \*m}) = (-q)^{I(\sigma_{\*m}) - I(\sigma_{\*r})} \qdet(X_{\widehat{\*m} \widehat{\*r}}) D^{-1},$$
    and therefore \eqref{eq:W qdet} becomes
    $$\sum_{\*m \in \II_k} (-q)^{I(\sigma_{\*m}) - I(\sigma_{\*r})} \qdet(X_{\widehat{\*m} \widehat{\*r}}) D^{-1} \qdet(X_{\*m \*c}) = \delta_{\*r \*c}.$$
    This is precisely the relation $(Y^{(n - k)})^{\mathsf{T}} D^{-1} X^{(k)} = I$.
\end{rem}

\section{Involutization of \texorpdfstring{$\aut^r(A_q)$}{aut(Aq)}}\label{sec:involutization}

Having described the Hopf algebra $\aut^r(A_q)$, we now move on to step \eqref{obj:cocommutative quotient} of the method outlined in Subsection \ref{subsec:observations}. In other words, we begin our classification of cocommutative quotients of $\aut^r(A_q)$. As mentioned in Subsection \ref{subsec:observations}, we start by ``involutizing'' the Hopf algebra $\aut^r(A_q)$. It is easy to see that every Hopf algebra has a maximal involutive quotient, which we call its \emph{involutization}. In other words, the ideal $(S_H^2(x) - x \mid x \in H)$ of a Hopf algebra $H$ is always a Hopf ideal.

\begin{dfn}\label{dfn:involutization}
    If $H$ is a Hopf algebra, we define its \emph{involutization} as
    $$H_{\mathrm{inv}} \coloneqq \frac{H}{(S_H^2(x) - x \mid x \in H)},$$
    where $S_H^2 = S_H \circ S_H$.
\end{dfn}

The main goal of this section is to explicitly compute the involutization of $\aut^r(A_q)$. To that end, we establish the following notation.

\begin{ntt}\label{ntt:universal involutive}
    Let $H_q(n) \coloneqq \aut^r(A_q(n))_{\mathrm{inv}}$. We often omit the $(n)$ if the value of $n$ is clear, and simply write $H_q$ instead of $H_q(n)$. We write $h_{ij}$ for the image of $x_{ij}$ in $H_q$, and we abuse notation by also using $D$ to denote the image of the quantum determinant of $\aut^r(A_q)$ in $H_q$. Furthermore, we define $\mathcal{H} \coloneqq (h_{ij}) \in M_n(H_q)$.
\end{ntt}

\begin{rem}
    The universal involutive coacting Hopf algebra of an Artin--Schelter regular algebra has been studied before -- see, for example, \cite[Definition 2.9(b)]{ChanWaltonZhang} and \cite[Definition 2.7(b)]{WaltonWang}. The Hopf algebra $H_q$ satisfies the following universal property: if $K$ is an involutive Hopf algebra which right coacts on $A_q$, then there is a unique homomorphism of Hopf algebras $H_q \to K$ such that the diagram
    \begin{center}
        \begin{tikzcd}
            & A \otimes H_q \arrow[d, dashed] \\
            A \arrow[ru] \arrow[r] & A \otimes K
        \end{tikzcd}
    \end{center}
    commutes.
\end{rem}

The following result gives an explicit description of $H_q$ by generators and relations.

\begin{prop}\label{prop:involutive Hopf algebra}
    Let $q \in \kk^*$. As an algebra, $H_q$ is generated by $h_{ij}$ for $1 \leq i,j \leq n$ and $D^{-1}$, subject to the following relations:
    \begin{subequations}\label{eq:general involutive relations}
        \begin{gather}
            h_{ik}h_{ij} = q h_{ij}h_{ik} \quad (k > j),\label{eq:general involutive q commute} \\
            ad - q^{-1}bc = \qdet\begin{pmatrix}
                a & b \\
                c & d
            \end{pmatrix} = da - q cb,\label{eq:general involutive 2x2 minor} \\
            h_{ij} D = q^{2(j - i)} D h_{ij}, \\
            \sum_{k = 1}^n (-q)^{i - k} \qdet(\mathcal{H}_{\hat k \hat\imath}) h_{k j} = \delta_{ij} D,\label{eq:general involutive long relation} \\
            D^{-1}D = DD^{-1} = 1,
        \end{gather}
    \end{subequations}
    for all $i,j,k \in \{1,\dots,n\}$, $\*r,\*c \in \II_k$, and all $2 \times 2$ sub-matrices $\begin{pmatrix}
        a & b \\
        c & d
    \end{pmatrix}$ of $\mathcal{H}$, where $D \coloneqq \qdet(\mathcal{H})$ is the quantum determinant of the matrix $\mathcal{H} = (h_{ij})$. The Hopf structure of $\aut^r(A_q)$ is given by
    \begin{align*}
        \Delta(h_{ij}) &= \sum_{k = 1}^n h_{ik} \otimes h_{kj}, \\
        \varepsilon(h_{ij}) &= \delta_{ij}, \\
        S(h_{ij}) &= (-q)^{j - i} \qdet(\mathcal{H}_{\hat{\jmath}\hat{\imath}})D^{-1}.
    \end{align*}
\end{prop}

To prove Proposition \ref{prop:involutive Hopf algebra}, we simply have to compute $S^2(x_{ij})$ for all $i,j$.

\begin{lem}\label{lem:computation of S^2}
    We have $S^2(x_{ij}) = q^{2(j - i)} D x_{ij} D^{-1}$.
\end{lem}
\begin{proof}
    We have $S(x_{ij}) = y_{ji}D^{-1}$, by Proposition \ref{prop:Y is inverse of X}. Since $S$ is an anti-homomorphism and $S(D^{-1}) = D$, we therefore have $S^2(x_{ij}) = D S(y_{ji})$. It follows from Lemma \ref{lem:U and V are invertible} and the proof of Theorem \ref{thm:description of universal Hopf algebra} that $S(y_{ji}) = q^{2(j - i)} x_{ij} D^{-1}$, which concludes the proof.
\end{proof}

Equipped with Lemma \ref{lem:computation of S^2}, the proof of Proposition \ref{prop:involutive Hopf algebra} follows easily.

\begin{proof}[Proof of Proposition \ref{prop:involutive Hopf algebra}]
    By Lemma \ref{lem:computation of S^2}, we have
    $$H_q = \frac{\aut^r(A_q)}{(x_{ij}D - q^{2(j - i)} D x_{ij})}.$$
    The result now follows immediately from Theorem \ref{thm:description of universal Hopf algebra}.
\end{proof}

\section{The two-variable case}\label{sec:two variables}

In this section, we consider the case where $n = 2$: we classify group gradings of the quantum plane
$$A_q(2) = \frac{\kk\ang{x,y}}{(yx - q xy)}.$$
Although this case has already been studied \cite{Crawford}, it is still useful to work out the details of our approach using Manin's universal quantum group, since this will lead to insights in the higher-dimensional cases. Furthermore, our method yields stronger results: on top of classifying group gradings of $A_q(2)$, we also classify all possible inner-faithful right coactions of cocommutative Hopf algebras on $A_q(2)$ by classifying the maximal cocommutative quotients of $\aut^r(A_q(2))$.

\subsection{Crawford's result}

Crawford proved that $A_q(2)$ can only be graded by abelian groups unless $q = -1$.

\begin{thm}[{\cite[Theorem 1.1]{Crawford}}]\label{thm:Simon}
    Suppose $A_q(2)$ has a faithful $G$-grading which refines its natural $\NN$-grading, where $G$ is a nonabelian group. Then $q = -1$ and $G$ is a quotient of the group $\ang{f,g \mid f^2 = g^2}$.
\end{thm}

Given Theorem \ref{thm:Simon}, we make the following definition.

\begin{ntt}\label{ntt:Simon group}
    We let $\Gamma \coloneqq \ang{f,g \mid f^2 = g^2}$.
\end{ntt}

\subsection{Cocommutative quotients of \texorpdfstring{$H_q(2)$}{Hq}}

We now proceed to classify all cocommutative quotients of $H_q(2)$. First, we apply Proposition \ref{prop:involutive Hopf algebra} to give an explicit presentation of $H_q(2)$. This presentation was also given in \cite[Lemma 5.3]{ChanWaltonZhang}.

\begin{cor}\label{cor:involutive relations 2 variables}
    Let $q \in \kk^*$. The Hopf algebra $H_q(2)$ has generators $h_{11}, h_{12}, h_{21}, h_{22}$, and $D^{-1}$, subject to the relations
    \begin{equation}\label{eq:involutive relations}
        \begin{gathered}
            \begin{aligned}
                h_{22} h_{11} &= h_{11} h_{22}, & h_{21} h_{12} &= q^{-2} h_{12} h_{21}, \\
                h_{12} h_{11} &= q h_{11} h_{12}, & h_{22} h_{21} &= qh_{21} h_{22}, \\
                h_{21} h_{11} &= q^{-1} h_{11} h_{21}, & h_{22} h_{12} &= q^{-1} h_{12} h_{22},
            \end{aligned} \\
            D D^{-1} = 1 = D^{-1} D,
        \end{gathered}
    \end{equation}
    where $D = h_{11}h_{22} - q^{-1}h_{12}h_{21}$ is the quantum determinant of $H_q(2)$.
\end{cor}
\begin{proof}
    The relations $h_{12} h_{11} = q h_{11} h_{12}$ and $h_{22} h_{21} = q h_{21} h_{22}$ come from \eqref{eq:general involutive q commute}.
    
    By \eqref{eq:general involutive long relation} with $i = 1$ and $j = 2$, we have
    $$h_{22} h_{12} - q^{-1} h_{12} h_{22} = 0,$$
    so we get the relation $h_{22} h_{12} = q^{-1} h_{12} h_{22}$. The relation $h_{21} h_{11} = q^{-1} h_{11} h_{21}$ follows similarly.

    Finally, by \eqref{eq:general involutive long relation} with $i = j = 1$, we get
    $$h_{22} h_{11} - q^{-1} h_{12} h_{21} = D.$$
    Recalling that $D = h_{11} h_{22} - q^{-1} h_{12} h_{21} = h_{22} h_{11} - q h_{21} h_{12}$, it follows that $h_{22} h_{11} = h_{11} h_{22}$ and that $h_{21} h_{12} = q^{-2} h_{12} h_{21}$.

    It is easy to check that the relations \eqref{eq:involutive relations} imply all the relations in \eqref{eq:general involutive relations}.
\end{proof}

\begin{rem}\label{rem:involutive coaction is compatible}
    The algebra $H_q(2)$ is also isomorphic to a deformation of $GL_2$ denoted by $GL_{q,q^{-1}}(2)$ which is given by Takeuchi in \cite{Takeuchi} (see also \cite[Proposition 5.4]{ChanWaltonZhang}).
\end{rem}

To state the classification of cocommutative Hopf coactions on $A_q(2)$, we must introduce the following family of Hopf algebras.

\begin{dfn}[{\cite[Construction 1.1]{GoodearlZhang}}]\label{dfn:Goodearl-Zhang Hopf algebra}
    For $n \in \ZZ$ and $q \in \kk^*$, define the Hopf algebra
    $$\mathcal{A}(n,q) \coloneqq \frac{\kk\ang{x^{\pm 1},y}}{(yx - qxy)}$$
    with Hopf structure given by
    \begin{align*}
        \Delta(x) &= x \otimes x, & \Delta(y) &= y \otimes 1 + x^n \otimes y, \\
        \varepsilon(x) &= 1, & \varepsilon(y) &= 0, \\
        S(x) &= x^{-1}, & S(y) &= -x^{-n}y.
    \end{align*}
    In other words, $x$ is group-like and $y$ is $x^n$-skew primitive.
\end{dfn}

Certainly, $\mathcal{A}(n,q)$ is cocommutative if and only if $n = 0$. As we will show in this section, the Hopf algebras $\mathcal{A}(0,q^{\pm 1})$ (and their quotients) are the only cocommutative Hopf algebras which coact on $A_q(2)$ inner-faithfully but are not group algebras.

\begin{thm}\label{thm:cocommutative quotient 2 variables}
    Let $q \in \kk^*$. Then, up to isomorphism, the maximal cocommutative quotients of $\aut^r(A_q)$ are the following.
    \begin{enumerate}
        \item The group algebra of $\ZZ^2$.
        \item The Hopf algebra $\mathcal{A}(0,q^{\pm 1})$ defined in Definition \ref{dfn:Goodearl-Zhang Hopf algebra}.
        \item The group algebra of $\Gamma$ (only if $q = -1$).
    \end{enumerate}
    Consequently, if $C$ is a cocommutative Hopf algebra which right coacts on $A_q(2)$ inner-faithfully, then $C$ is a quotient of one of the Hopf algebras on the list above.
\end{thm}

Note that when we say ``maximal cocommutative quotient'', this is meant in the sense of Proposition \ref{prop:maximal cocommutative quotients}. As an immediate consequence of Theorem \ref{thm:cocommutative quotient 2 variables}, we recover Crawford's result (Theorem \ref{thm:Simon}).

For the rest of this section, we fix a Hopf ideal $I$ of $H_q(2)$ such that $\overline{H} \coloneqq H_q(2)/I$ is cocommutative. We write $g_{ij} \coloneqq h_{ij} + I$ for the images of the generators $h_{ij}$ in $\overline{H}$, and $\overline{D} \coloneqq D + I$ for the image of the quantum determinant in $\overline{H}$. The following lemma analyzes the consequences of requiring $\overline{H}$ to be cocommutative.

\begin{lem}\label{lem:linear dependence 2 variables}
    We have the following conditions:
    \begin{enumerate}
        \item $g_{12}$ and $g_{21}$ are linearly dependent.\label{cond:linear dependence of b and c}
        \item $g_{12}$ and $g_{11} - g_{22}$ are linearly dependent.\label{cond:b is proportional to a - d}
        \item $g_{21}$ and $g_{11} - g_{22}$ are linearly dependent.\label{cond:c is proportional to a - d}
    \end{enumerate}
\end{lem}
\begin{proof}
    In this proof, we use the standard fact that if $V$ is a vector space and $u,v \in V$, then $u \otimes v = v \otimes u$ in $V \otimes V$ if and only if $u$ and $v$ are linearly dependent.

    The coproduct on $g_{11}$ is
    $$\Delta(g_{11}) = g_{11} \otimes g_{11} + g_{12} \otimes g_{21}.$$
    By the cocommutativity of $\overline{H}$, it follows that
    $$g_{11} \otimes g_{11} + g_{12} \otimes g_{21} = g_{11} \otimes g_{11} + g_{21} \otimes g_{12},$$
    and therefore
    $$g_{12} \otimes g_{21} = g_{21} \otimes g_{12}.$$
    Condition \eqref{cond:linear dependence of b and c} follows.
    
    For \eqref{cond:b is proportional to a - d}, consider
    $$\Delta(g_{12}) = g_{11} \otimes g_{12} + g_{12} \otimes g_{22}.$$
    Since $\overline{H}$ is cocommutative, we have
    $$g_{11} \otimes g_{12} + g_{12} \otimes g_{22} = g_{12} \otimes g_{11} + g_{22} \otimes g_{12},$$
    and therefore,
    $$(g_{11} - g_{22}) \otimes g_{12} = g_{12} \otimes (g_{11} - g_{22}).$$
    Condition \eqref{cond:b is proportional to a - d} follows. Condition \eqref{cond:c is proportional to a - d} is obtained similarly by considering $\Delta(g_{21})$.
\end{proof}

First, we consider what happens when $g_{12} = g_{21} = 0$.

\begin{lem}\label{lem:commutative quotient}
    The ideal $I = (h_{12},h_{21})$ of $H_q(2)$ is a Hopf ideal, and $H_q(2)/I \cong \kk\ZZ^2$ as Hopf algebras.
\end{lem}
\begin{proof}
    We omit proving that $I$ is a Hopf ideal, since this is straightforward. Adopting the notation $\overline{H}$, $g_{ij}$, and $\overline{D}$ from above, we know that $g_{11}$ and $g_{22}$ commute by Corollary \ref{cor:involutive relations 2 variables}. Furthermore, we can easily see that $g_{11}$ and $g_{22}$ are group-like, and therefore are invertible in $\overline{H}$. It is now clear that $\overline{H} \cong \kk \ZZ^2$ by $g_{11} \mapsto (1,0)$ and $g_{22} \mapsto (0,1)$.
\end{proof}

Therefore, it remains to consider the cases where at least one of $g_{12}$ or $g_{21}$ is nonzero.

\begin{lem}\label{lem:diagonals are equal}
    Assume that $q \neq 1$. Suppose that $g_{12} \neq 0$ or $g_{21} \neq 0$. Then $g_{11} = g_{22}$.
\end{lem}
\begin{proof}
    First assume that $g_{12} \neq 0$. By Lemma \ref{lem:linear dependence 2 variables}, we must have $g_{11} - g_{22} = \lambda g_{12}$ for some $\lambda \in \kk$. To prove the result, it suffices to show that $\lambda = 0$. Considering the antipode on $\overline{H}$, we have
    $$S(g_{11} - g_{22}) = (g_{22} - g_{11})\overline{D}^{-1} = -\lambda g_{12} \overline{D}^{-1},$$
    where we used that $g_{11} - g_{22} = \lambda g_{12}$ in the last equality. On the other hand,
    $$S(g_{11} - g_{22}) = \lambda S(g_{12}) = -\lambda q g_{12} \overline{D}^{-1},$$
    and therefore $(1 - q) \lambda g_{12} = 0$, which implies that $\lambda = 0$ (since $q \neq 1$), and therefore $g_{11} = g_{22}$. The case $g_{21} \neq 0$ follows similarly.
\end{proof}

We now consider the case where exactly one of $g_{12}$ or $g_{21}$ is zero.

\begin{prop}\label{prop:non-group cocommutative quotient}
    The ideals $I_1 = (h_{11} - h_{22}, h_{21})$ and $I_2 = (h_{11} - h_{22}, h_{12})$ are Hopf ideals of $H_q(2)$, and $H_q(2)/I_1 \cong \mathcal{A}(0,q)$ and $H_q(2)/I_2 \cong \mathcal{A}(0,q^{-1})$ as Hopf algebras.
\end{prop}
\begin{proof}
    We only prove the result for $I_1$, since the proof for $I_2$ is similar. To that end, we let $I \coloneqq I_1$, and adopt the notation from above: we let $\overline{H} = H_q(2)/I$, $g_{ij} = h_{ij} + I$, and $\overline{D} = D + I$.
    
    First, we prove that $I$ is a Hopf ideal of $H_q(2)$. It is clear that $\varepsilon(I) = 0$. Let $\Phi$ be the composition
    $$\Phi \colon H_q(2) \xrightarrow{\Delta} H_q(2) \otimes H_q(2) \surj \overline{H} \otimes \overline{H}.$$
    To prove that $\Delta(I) \subseteq H_q(2) \otimes I + I \otimes H_q(2)$, it suffices to show that $\Phi(I) = 0$. We have
    $$\Phi(h_{21}) = g_{21} \otimes g_{11} + g_{22} \otimes g_{21} = 0,$$
    since $g_{21} = 0$ in $\overline{H}$. Furthermore,
    $$\Phi(h_{11} - h_{22}) = g_{11} \otimes g_{11} + g_{12} \otimes g_{21} - (g_{21} \otimes g_{12} + g_{22} \otimes g_{22}) = g_{11} \otimes g_{11} - g_{11} \otimes g_{11} = 0,$$
    where we used that $g_{21} = 0$ and $g_{11} = g_{22}$. It follows that $\Delta(I) \subseteq H_q(2) \otimes I + I \otimes H_q(2)$, as required.

    Finally, we check the antipode:
    $$S(h_{21}) = -q^{-1} h_{21} D^{-1} \in I, \quad S(h_{11} - h_{22}) = (h_{22} - h_{11})D^{-1} \in I,$$
    and thus $S(I) \subseteq I$. We conclude that $I = I_1 = (h_{11} - h_{22}, h_{21})$ is a Hopf ideal of $H_q(2)$.

    We now prove the isomorphism in the statement of the proposition. The Hopf algebra $\overline{H}$ is generated by $g_{11}$, $g_{12}$, and $\overline{D}^{-1}$, subject to the relation $g_{12}g_{11} = q g_{11}g_{12}$, by Corollary \ref{cor:involutive relations 2 variables}. By definition of the quantum determinant $D$, we have $\overline{D} = g_{11}^2$, so $g_{11}$ is invertible in $\overline{H}$. Note that
    $$\Delta(g_{11}^{-1} g_{12}) = 1 \otimes (g_{11}^{-1} g_{12}) + (g_{11}^{-1} g_{12}) \otimes 1,$$
    and thus $g_{11}^{-1} g_{12}$ is primitive. We therefore get an isomorphism of Hopf algebras
    \begin{align*}
        \mathcal{A}(0,q) &\xrightarrow{\sim} \overline{H} \\
        x &\mapsto g_{11}, \\
        y &\mapsto g_{11}^{-1}g_{12}. \qedhere
    \end{align*}
\end{proof}

The last case to consider is when $g_{12}$ and $g_{21}$ are both nonzero in $\overline{H}$. As we prove next, this can only happen if $q^2 = 1$.

\begin{lem}\label{lem:b and c nonzero}
    Suppose $g_{12}, g_{21} \neq 0$. Then $q^2 = 1$.
\end{lem}
\begin{proof}
    Assume, for a contradiction, that $q^2 \neq 1$. Lemma \ref{lem:linear dependence 2 variables} implies that $g_{21} = \lambda g_{12}$ for some $\lambda \in \kk^*$. By \eqref{eq:involutive relations}, we have $g_{12}g_{21} = q^2 g_{21}g_{12}$. Since $g_{21} = \lambda g_{12}$ and $\lambda \neq 0$, it follows that $g_{12}^2 = q^2 g_{12}^2$. But $q^2 \neq 1$ by assumption, so $g_{12}^2 = 0$.
    
    Similarly, we have $q g_{21}g_{11} = g_{11}g_{21}$. On the other hand, using that $g_{21} = \lambda g_{12}$, we get
    $$q g_{21}g_{11} = \lambda q g_{12}g_{11} = \lambda q^2 g_{11}g_{12} = q^2 g_{11}g_{21},$$
    and thus $g_{11}g_{21} = q^2 g_{11}g_{21}$. But $q^2 \neq 1$, so $g_{11}g_{21} = 0$.

    Note that by Lemma~\ref{lem:diagonals are equal}, $g_{11} = g_{22}$.  Thus by definition of $D$, we have 
    $$\overline{D} = g_{11}g_{22} - q^{-1}g_{12}g_{21} = g_{11}^2 - \lambda q^{-1} g_{12}^2 = g_{11}^2.$$
    Therefore, $g_{11}$ is invertible in $\overline{H}$. But then $g_{11}g_{21} = 0$ implies that $g_{21} = 0$ in $\overline{H}$, a contradiction.
\end{proof}

As the next result shows, if $q = -1$ and $g_{12}, g_{21} \neq 0$, then the only possibility is that $\overline{H}$ is a quotient of $\kk \Gamma$.

\begin{prop}\label{prop:nonabelian group quotient}
    Let $\lambda \in \kk^*$. Then the ideal $I = (h_{11} - h_{22}, h_{21} - \lambda h_{12})$ of $H_{-1}(2)$ is a Hopf ideal and $H_{-1}(2)/I \cong \kk\Gamma$ as Hopf algebras.
\end{prop}
\begin{proof}
    We preserve the notation $\overline{H} = H_{-1}(2)/I$, $g_{ij} = h_{ij} + I$, and $\overline{D} = D + I$ from above. First, it is clear that $\varepsilon(I) = 0$. As before, we let $\Phi$ be the composition
    $$\Phi \colon H_{-1}(2) \xrightarrow{\Delta} H_{-1}(2) \otimes H_{-1}(2) \surj \overline{H} \otimes \overline{H}.$$
    We have
    \begin{align*}
        \Phi(h_{11} - h_{22}) &= g_{11} \otimes g_{11} + g_{12} \otimes g_{21} - (g_{21} \otimes g_{12} + g_{22} \otimes g_{22}) \\
        &= g_{11} \otimes g_{11} + g_{12} \otimes (\lambda g_{12}) - (\lambda g_{12}) \otimes g_{12} - g_{11} \otimes g_{11} = 0,
    \end{align*}
    where we used that $g_{22} = g_{11}$ and $g_{21} = \lambda g_{12}$. Similarly,
    \begin{align*}
        \Phi(h_{21} - \lambda h_{12}) &= g_{21} \otimes g_{11} + g_{22} \otimes g_{21} - \lambda(g_{11} \otimes g_{12} + g_{12} \otimes g_{22}) \\
        &= (\lambda g_{12}) \otimes g_{11} + g_{11} \otimes (\lambda g_{12}) - \lambda g_{11} \otimes g_{12} - \lambda g_{12} \otimes g_{11} = 0.
    \end{align*}
    Checking the antipode, we get
    $$S(h_{11} - h_{22}) = (h_{22} - h_{11})D^{-1} \in I, \quad S(h_{21} - \lambda h_{12}) = (h_{21} - \lambda h_{12})D^{-1} \in I,$$
    and therefore $S(I) \subseteq I$, so $I$ is a Hopf ideal of $H_{-1}(2)$. It follows from Corollary \ref{cor:involutive relations 2 variables} that $\overline{H}$ is generated by $g_{11}$, $g_{12}$, and $\overline{D}^{-1}$ subject to the relation $g_{12}g_{11} = -g_{11}g_{12}$. In this case, we have $\overline{D} = g_{11}^2 + \lambda g_{12}^2$. Letting $\mu \coloneqq \sqrt{\lambda}$, we see that
    $$(g_{11} + \mu g_{12})^2 = (g_{11} - \mu g_{12})^2 = g_{11}^2 + \lambda g_{12}^2 = \overline{D}.$$
    Therefore, letting $u \coloneqq g_{11} + \mu g_{12}$ and $v \coloneqq g_{11} - \mu g_{12}$, it follows that $\overline{D} = u^2 = v^2$, and thus $u$ and $v$ are both invertible in $\overline{H}$. Hence,
    $$\overline{H} = \frac{\kk\ang{u^{\pm 1},v^{\pm 1}}}{(u^2 - v^2)}.$$
    It is straightforward to check that $u$ and $v$ are group-like, so $G(\overline{H}) = \ang{u,v \mid u^2 = v^2} \cong \Gamma$, and thus $\overline{H} \cong \kk \Gamma$.
\end{proof}

We are now ready to prove Theorem \ref{thm:cocommutative quotient 2 variables}.

\begin{proof}[Proof of Theorem \ref{thm:cocommutative quotient 2 variables}]
    Let $C$ be a cocommutative Hopf algebra and let $\varphi \colon H_q(2) \to C$ be a Hopf algebra homomorphism. Write $g_{ij} \coloneqq \varphi(h_{ij})$. If $g_{12} = g_{21} = 0$, then Lemma \ref{lem:commutative quotient} implies that $\varphi$ factors through $\kk \ZZ^2$. So, suppose at least one of $g_{12}$ or $g_{21}$ is nonzero.
    
    First, we assume that $q \neq 1$. Then $g_{11} = g_{22}$, by Lemma \ref{lem:diagonals are equal}. Proposition \ref{prop:non-group cocommutative quotient} implies that if $g_{21} = 0$ then $\varphi$ factors through $\mathcal{A}(0,q)$, while if $g_{12} = 0$ then $\varphi$ factors through $\mathcal{A}(0,q^{-1})$.
    
    Next, if $g_{12}, g_{21} \neq 0$, then Lemma \ref{lem:b and c nonzero} implies that $q = -1$. By Lemma~\ref{lem:linear dependence 2 variables}\eqref{cond:linear dependence of b and c}, $g_{21} = \lambda g_{12}$ for some $\lambda \in \kk^*$. Now it follows by Proposition \ref{prop:nonabelian group quotient} that $\varphi$ factors through $\kk\Gamma$.

    Finally, we now consider the case $q = 1$. Since $H_1(2)$ is commutative, it follows that $\overline{H} \coloneqq \im(\varphi)$ is commutative. By Lemma \ref{lem:linear dependence 2 variables}, every pair of $g_{12}, g_{21}$, and $g_{11} - g_{22}$ is linearly dependent. Assume that $g_{12} \neq 0$, since the other case is similar. Now we can get $g_{21} = \lambda g_{12}$ and $g_{22} = g_{11} + \mu g_{12}$ for some $\lambda, \mu \in \kk$. In other words,
    $$\begin{pmatrix}
        g_{11} & g_{12} \\
        g_{21} & g_{22}
    \end{pmatrix} = \begin{pmatrix}
        g_{11}  & g_{12} \\
        \lambda g_{12} & g_{11} + \mu g_{12}
    \end{pmatrix}.$$
    For any matrix $M = \begin{pmatrix} a & b \\ c & d \end{pmatrix} \in GL_2(\kk)$, we have an automorphism of $\kk[x_1, x_2]$ with 
    \[
    \begin{pmatrix}
        x_1 \\
        x_2
    \end{pmatrix} \mapsto \begin{pmatrix} 
        a & b \\
        c & d
    \end{pmatrix} \begin{pmatrix}
        x_1 \\
        x_2
    \end{pmatrix} = \begin{pmatrix}
        ax_1 + bx_2 \\
        cx_1 + dx_2
    \end{pmatrix}.
    \]
    This induces an automorphism of the universal coacting Hopf algebra via
    \[
    \begin{pmatrix}
        h_{11} & h_{12} \\
        h_{21} & h_{22}
    \end{pmatrix} \mapsto \begin{pmatrix}
        h'_{11} & h'_{12} \\
        h'_{21} & h'_{22}
    \end{pmatrix} \coloneqq M^{-1} \begin{pmatrix}
        h_{11} & h_{12} \\
        h_{21} & h_{22}
    \end{pmatrix} M
    \]
    We claim that we can choose $M$ so that $\varphi(h_{21}') = 0$. Define $\alpha \coloneqq \frac{\mu + \sqrt{\mu^2 + 4\lambda}}{2}$ and $\beta \coloneqq \frac{\mu - \sqrt{\mu^2 + 4\lambda}}{2}$, and let $M \coloneqq \begin{pmatrix} 
        1 & 0 \\
        \alpha & 1
    \end{pmatrix}$. Then
    $$M^{-1} \begin{pmatrix}
        g_{11} & g_{12} \\
        g_{21} & g_{22}
    \end{pmatrix} M = \begin{pmatrix}
        g_{11} + \alpha g_{12} & g_{12} \\
        0 & g_{11} + \beta g_{12}
    \end{pmatrix}.$$
    Let $a \coloneqq g_{11} + \alpha g_{12}$, $b \coloneqq g_{12}$, and $c \coloneqq g_{11} + \beta g_{12}$. Using the fact that $\alpha^2 = \alpha \mu + \lambda$ and $\beta^2 = \beta \mu + \lambda$, we can easily see that
    $$\Delta(a) = a \otimes a, \qquad \Delta(c) = c \otimes c.$$
    In other words, both $a$ and $c$ are group-like. Note that, if $\alpha \neq \beta$ (equivalently, $\mu^2 + 4\lambda \neq 0$), then $a$ and $c$ are sufficient to generate $\overline{H}$. Otherwise, $\overline{H}$ is generated by $a$ and $b$ (since $a = c$ in this case). As a result, the rest of the proof is split into two cases.

    \begin{case}
        $\mu^2 + 4\lambda \neq 0$.
    \end{case}

    As mentioned above, in this case the elements $a$ and $c$ generate the Hopf algebra $\overline{H}$, so $\overline{H}$ is generated by group-like elements. Therefore, $\overline{H}$ is an abelian group algebra, and thus is a quotient of $\kk\ZZ^2$.

    \begin{case}
        $\mu^2 + 4\lambda = 0$.
    \end{case}

    In this case, we have $\mu = 2\alpha$, so it is straightforward to check that $\Delta(b) = a \otimes b + b \otimes a$. It follows that $a^{-1}b$ is primitive, and therefore $\overline{H}$ is a quotient of $\mathcal{A}(0,1)$ via the map $x \mapsto a$ and $y \mapsto a^{-1} b$.
\end{proof}

\section{The three-variable case}\label{sec:three variables}

We now move on to the case where $n = 3$: we classify group gradings of the three-variable skew polynomial ring
$$A_q(3) = \frac{\kk\ang{x_1,x_2,x_3}}{(x_j x_i - q x_i x_j \mid j > i)}.$$
Just like we did in the two-variable case, we also classify all possible inner-faithful coactions of Hopf algebras on $A_q(3)$ (but only for $q^2 \neq 1$).

\subsection{The universal involutive Hopf algebra}

We now explicitly describe the Hopf algebra $H_q(3)$ in the three-variable case. By \eqref{eq:general involutive relations}, $H_q(3)$ has the following relations:
\begin{subequations}\label{eq:involutive relations 3 variables}
    \begin{gather}
        h_{ik}h_{ij} = q h_{ij}h_{ik} \quad (k > j),\label{eq:involutive q commute} \\
        ad - q^{-1}bc = \qdet\begin{pmatrix}
            a & b \\
            c & d
        \end{pmatrix} = da - q cb,\label{eq:involutive 2x2 minor} \\
        h_{ij} D = q^{2(j - i)}D h_{ij},\label{eq:involutive qdet commutation} \\
        \sum_{k = 1}^3 (-q)^{i - k} \qdet(\mathcal{H}_{\hat{k}\hat{\imath}}) h_{kj} = \delta_{ij} D,\label{eq:involutive long relation} \\
        D^{-1}D = DD^{-1} = 1,
    \end{gather}
\end{subequations}
for all $i,j,k \in \{1,2,3\}$ and all $2 \times 2$ sub-matrices $\begin{pmatrix}
    a & b \\
    c & d
\end{pmatrix}$ of $\mathcal{H}$, where the quantum determinant of $H_q(3)$ is
$$D = h_{11}h_{22}h_{33} - q^{-1} h_{11}h_{23}h_{32} - q^{-1} h_{12}h_{21}h_{33} + q^{-2} h_{12}h_{23}h_{31} + q^{-2} h_{13}h_{21}h_{32} - q^{-3} h_{13}h_{22}h_{31}.$$

\subsection{Cocommutative quotients of \texorpdfstring{$H_q(3)$}{Hq}}

Fix a Hopf ideal $I$ of $H_q(3)$ such that $\overline{H} \coloneqq H_q(3)/I$ is cocommutative. We write $g_{ij} \coloneqq h_{ij} + I$ for the images of the generators $h_{ij}$ in $\overline{H}$, and $\overline{D} \coloneqq D + I$ for the image of the quantum determinant in $\overline{H}$.

To state the classification result, we introduce two cocommutative Hopf algebras which are not group algebras. Both of these are extensions of the Hopf algebra $\mathcal{A}(0,q)$ from Definition \ref{dfn:Goodearl-Zhang Hopf algebra}.

\begin{dfn}\label{dfn:Bq and Cq}
    For $q \in \kk^*$, we define Hopf algebras
    $$\mathcal{B}_q \coloneqq \frac{\kk\ang{x^{\pm 1},y^{\pm 1},z}}{(yx - xy, zx - qxz, zy - yz)}, \quad \mathcal{C}_q \coloneqq \frac{\kk\ang{x^{\pm 1},y^{\pm 1},z}}{(yx - xy, zx - qxz, zy - q^2 yz)},$$
    with Hopf structures given by $x$ and $y$ being group-like and $z$ being primitive.
\end{dfn}
It is straightforward to check directly that $\mathcal{B}_q$ and $\mathcal{C}_q$ are Hopf algebras, but in any case this will follow from Proposition~\ref{prop:generic quotient Bq or Cq} below.

The following is the main result of this section, which classifies all inner-faithful cocommutative coactions on $A_q(3)$ for $q \neq \pm 1$. Just like the two-variable case, we only get coactions by abelian groups in this case, and also by $\mathcal{B}_{q^{\pm 1}}$ and $\mathcal{C}_{q^{\pm 1}}$, analogously to the coaction by $\mathcal{A}(0,q^{\pm 1})$ in the two-variable case.

\begin{thm}\label{thm:generic}
    Let $q \in \kk^* \setminus \{\pm 1\}$. Then, up to isomorphism, the maximal cocommutative quotients of $\aut^r(A_q)$ are the following.
    \begin{enumerate}
        \item The group algebra of $\ZZ^3$.
        \item The Hopf algebra $\mathcal{B}_{q^{\pm 1}}$.
        \item The Hopf algebra $\mathcal{C}_{q^{\pm 1}}$.
    \end{enumerate}
    Consequently, if $C$ is a cocommutative Hopf algebra which right coacts on $A_q(3)$ inner-faithfully, then $C$ is a quotient of one of the Hopf algebras on the list above. In particular, $A_q(3)$ does not have a faithful grading by a nonabelian group.
\end{thm}

First, we show how the Hopf algebras $\mathcal{B}_{q^{\pm 1}}$ and $\mathcal{C}_{q^{\pm 1}}$ arise as quotients of $H_q(3)$.

\begin{prop}\label{prop:generic quotient Bq or Cq}
    Let $i,j \in \{1,2,3\}$ with $i \neq j$, and define
    $$I_{ij} \coloneqq (h_{ii} - h_{jj}, h_{k\ell} \mid k \neq \ell, (k,\ell) \neq (i,j))$$
    as an ideal of $H_q(3)$. Then $I_{ij}$ is a Hopf ideal of $H_q(3)$, and
    $$H_q(3)/I_{ij} \cong \begin{cases}
        \mathcal{B}_q, &\text{if } j = i + 1, \\
        \mathcal{B}_{q^{-1}}, &\text{if } i = j + 1, \\
        \mathcal{C}_q, &\text{if } (i,j) = (1,3), \\
        \mathcal{C}_{q^{-1}}, &\text{if } (i,j) = (3,1),
    \end{cases}$$
    where $\mathcal{B}_q$ and $\mathcal{C}_q$ are defined in Definition \ref{dfn:Bq and Cq}.
\end{prop}
\begin{proof}
    It is straightforward to prove that $I_{ij}$ is a Hopf ideal of $H_q(3)$ (the proof is similar to that of Proposition \ref{prop:non-group cocommutative quotient}). We now let $I = I_{ij}$ and adopt the notation from above: we have $\overline{H} = H_q(3)/I$, $g_{ij} = h_{ij} + I$, and $\overline{D} = D + I$. We will only prove the result in the cases $(i,j) = (1,2)$ and $(i,j) = (1,3)$, since the other cases are similar.

    \begin{case}\label{case:g12 isomorphism to Bq}
        $(i,j) = (1,2)$.
    \end{case}

    The Hopf algebra $\overline{H}$ is generated by $g_{11}, g_{33}, g_{12}$, and $\overline{D}^{-1}$, subject to the relations
    \begin{alignat*}{2}
        g_{12} g_{11} &= q g_{11} g_{12} \quad & &\text{by } \eqref{eq:involutive q commute} \text{ with } i = j = 1, k = 2, \\
        g_{33}g_{11} &= g_{11}g_{33} \quad & &\text{by } \eqref{eq:involutive 2x2 minor} \text{ with } \begin{pmatrix}
            a & b \\
            c & d
        \end{pmatrix} = \begin{pmatrix}
            h_{11} & h_{13} \\
            h_{31} & h_{33}
        \end{pmatrix}, \\
        g_{12}g_{33} &= g_{33}g_{12} \quad & &\text{by } \eqref{eq:involutive 2x2 minor} \text{ with } \begin{pmatrix}
            a & b \\
            c & d
        \end{pmatrix} = \begin{pmatrix}
            h_{12} & h_{13} \\
            h_{32} & h_{33}
        \end{pmatrix}.
    \end{alignat*}
    The other relations from \eqref{eq:involutive relations 3 variables} do not result in any additional relations in $\overline{H}$. We also have $\overline{D} = g_{11}^2 g_{33}$, and thus $g_{11}$ and $g_{33}$ are invertible in $\overline{H}$ (in fact, it is easy to check that $g_{11}$ and $g_{33}$ are group-like). Furthermore,
    $$\Delta(g_{11}^{-1} g_{12}) = 1 \otimes (g_{11}^{-1} g_{12}) + (g_{11}^{-1} g_{12}) \otimes 1,$$
    so we see that $g_{11}^{-1} g_{12}$ is primitive. It follows that
    \begin{align*}
        \mathcal{B}_q &\to \overline{H} \\
        x &\mapsto g_{11}, \\
        y &\mapsto g_{33}, \\
        z &\mapsto g_{11}^{-1} g_{12}
    \end{align*}
    is an isomorphism of Hopf algebras.

    \begin{case}
        $(i,j) = (1,3)$.
    \end{case}

    This is similar to Case \ref{case:g12 isomorphism to Bq}, except $\overline{H}$ is generated by the elements $g_{11}, g_{22}, g_{13}$, and $\overline{D}^{-1}$, and the relations are slightly different:
    \begin{alignat*}{2}
        g_{13} g_{11} &= q g_{11} g_{13} \quad & &\text{by } \eqref{eq:involutive q commute} \text{ with } i = j = 1, k = 3, \\
        g_{22}g_{11} &= g_{11}g_{22} \quad & &\text{by } \eqref{eq:involutive 2x2 minor} \text{ with } \begin{pmatrix}
            a & b \\
            c & d
        \end{pmatrix} = \begin{pmatrix}
            h_{11} & h_{12} \\
            h_{21} & h_{22}
        \end{pmatrix}, \\
        g_{13}g_{22} &= q^2 g_{22}g_{13} \quad & &\text{by } \eqref{eq:involutive 2x2 minor} \text{ with } \begin{pmatrix}
            a & b \\
            c & d
        \end{pmatrix} = \begin{pmatrix}
            h_{12} & h_{13} \\
            h_{22} & h_{23}
        \end{pmatrix}.
    \end{alignat*}
    Similarly to Case \ref{case:g12 isomorphism to Bq}, we now see that the map
    \begin{align*}
        \mathcal{C}_q &\to \overline{H} \\
        x &\mapsto g_{11}, \\
        y &\mapsto g_{22}, \\
        z &\mapsto g_{11}^{-1}g_{13}
    \end{align*}
    is an isomorphism of Hopf algebras.

    The other cases follow similarly with very minor changes.
\end{proof}

\subsubsection{Consequences of \texorpdfstring{$\overline{H}$}{H} being a Hopf algebra}

As a first step in the proof of Theorem \ref{thm:generic}, we analyze some consequences of the requirement that $\overline{H}$ is a Hopf algebra. We start with an observation which follows easily by the structure of the comultiplication on $\overline{H}$.

\begin{lem}\label{lem:Hopf ideal}
    Choose $i,j,k$ such that $\{i,j,k\} = \{1,2,3\}$. If $g_{ij} = 0$, then either $g_{ik} = 0$ or $g_{kj} = 0$.
\end{lem}
\begin{proof}
    We have
    $$0 = \Delta(g_{ij}) = \sum_{\ell = 1}^3 g_{i\ell} \otimes g_{\ell j} = g_{ik} \otimes g_{kj}.$$
    The result follows.
\end{proof}

The second observation is that when $g_{ij} = 0$, we can get some further restrictions by referring to the $n = 2$ case. For example, if $g_{12} = 0$, then Lemma \ref{lem:Hopf ideal} implies that $g_{13} = 0$ or $g_{32} = 0$. In the first case, the first row of the matrix $\mathcal{G} = (g_{ij})$ has two zeros and in the second case, the second column of $\mathcal{G}$ has two zeros. In other words,
$$\textbf{Case 1: } \mathcal{G} = \begin{pmatrix}
    g_{11} & 0 & 0 \\
    g_{21} & g_{22} & g_{23} \\
    g_{31} & g_{32} & g_{33}
\end{pmatrix}, \qquad \textbf{Case 2: } \mathcal{G} = \begin{pmatrix}
    g_{11} & 0 & g_{13} \\
    g_{21} & g_{22} & g_{23} \\
    g_{31} & 0 & g_{33}
\end{pmatrix}.$$
As we show in the next proposition, the subalgebra of $\overline{H}$ generated by the entries of $\mathcal{G}_{\hat{1} \hat{1}}$ (in Case 1) or $\mathcal{G}_{\hat{2} \hat{2}}$ (in Case 2) is a Hopf quotient of $H_q(2)$, allowing us to use Theorem \ref{thm:cocommutative quotient 2 variables}.  Before proceeding, it will be useful to have the following calculation at hand.

\begin{lem}\label{lem:easy calculation}
    Choose $i,j,k$ such that $\{i,j,k\} = \{1,2,3\}$.  If $g_{ik} = 0$ or $g_{ji} = 0$, then 
    $$
    g_{ii}g_{jk} = q^{\varepsilon} g_{jk} g_{ii},
    $$
    where $\varepsilon = \begin{cases} 
        0, & i \in \{1, 3 \}, \\
        -2, & (i,j,k) = (2, 1, 3), \\
        2, & (i,j,k) = (2,3,1).
    \end{cases}$
\end{lem}
\begin{proof}
    The relation \eqref{eq:involutive 2x2 minor} applied to the $2 \times 2$ minor $\mathcal{G}_{\hat{k}, \hat{j}}$ has the form 
    $$ g_{ii} g_{jk} - q^{\pm 1} g_{ik} g_{ji} = g_{jk}g_{ii} - q^{\mp 1} g_{ji} g_{ik}$$
    when $i \in \{1, 3 \}$, the form
    $$   g_{ji} g_{ik} -q^{-1} g_{jk} g_{ii} = g_{ik} g_{ji} - q g_{ii} g_{jk}$$
    when $(i,j,k) = (2, 1, 3)$, and the form 
    $$   g_{ik} g_{ji}  -q^{-1} g_{ii} g_{jk}  = g_{ji} g_{ik} - q g_{jk} g_{ii}$$
    when $(i,j,k) = (2, 3, 1)$.
    The result follows immediately since $g_{ik} g_{ji} = g_{ji}g_{ik} = 0$.
\end{proof}

\begin{prop}\label{prop:Hopf subalgebra}
    Choose $i,j,k$ such that $\{i,j,k\} = \{1,2,3\}$ and suppose $g_{ij} = 0$, so that $g_{ik} = 0$ or $g_{kj} = 0$ by Lemma \ref{lem:Hopf ideal}. Let $\mathcal{G} \coloneqq (g_{ij}) \in M_3(\overline{H})$. Then the following hold.
    \begin{enumerate}
        \item Suppose $g_{ik} = 0$ and let
        $$D' \coloneqq \qdet(\mathcal{G}_{\hat{\imath} \hat{\imath}}) = g_{jj} g_{kk} - q^{\varepsilon} g_{jk} g_{kj},$$
        where $\varepsilon = -1$ when  $k > j$ and $\varepsilon = 1$ when $j > k$. Then $D'$ is invertible in $\overline{H}$ and the subalgebra $H' \coloneqq \kk\ang{g_{jj}, g_{kk}, g_{jk}, g_{kj}, (D')^{-1}}$ is a Hopf subalgebra of $\overline{H}$ which is a Hopf quotient of $H_q(2)$.\label{item:g_ik = 0}
        \item Suppose $g_{kj} = 0$ and let
        $$D' \coloneqq \qdet(\mathcal{G}_{\hat{\jmath} \hat{\jmath}}) = g_{ii} g_{kk} - q^{\varepsilon} g_{ik} g_{ki},$$
        where $\varepsilon = -1$ when $k > i$ and $\varepsilon = -1$ when $i > k$.   Then $D'$ is invertible in $\overline{H}$ and the subalgebra $H' \coloneqq \kk\ang{g_{ii}, g_{kk}, g_{ik}, g_{ki}, (D')^{-1}}$ is a Hopf subalgebra of $\overline{H}$ which is a Hopf quotient of $H_q(2)$.\label{item:g_kj = 0}
    \end{enumerate}
\end{prop}
\begin{proof}
    For \eqref{item:g_ik = 0}, we set $r \coloneqq i$, and for \eqref{item:g_kj = 0}, we set $r \coloneqq j$. 
    Let $n \coloneq \min(\{1,2,3\} \setminus \{r\})$ and $m \coloneqq \max(\{1,2,3\} \setminus \{r\})$.  We either have $g_{rn} = g_{rm} = 0$ or $g_{nr} = g_{mr} = 0$. Thus, using the relations \eqref{eq:involutive 2x2 minor} is it easy to check that $g_{rr}$ commutes with both $g_{mm}$ and $g_{nn}$. Moreover, by Lemma~\ref{lem:easy calculation} we see that $g_{rr}$ commutes with $g_{mn}g_{nm}$ and $g_{nm} g_{mn}$, and thus commutes with $D'$. Now a direct calculation using the formula for $D$ shows that $\overline{D} = g_{rr} D' = D' g_{rr}$. Hence, $g_{rr}$ and $D'$ are units in $\overline{H}$.

    We now check that $H'$ satisfies all the relations of $H_q(2)$ from \eqref{eq:involutive relations}.  We already know that the generators $g_{nn}, g_{mm}, g_{nm}, g_{mn}$ of $H'$ satisfy
    $$g_{nm} g_{nn} = q g_{nn} g_{nm}, \qquad g_{mm} g_{mn} = q g_{mn} g_{mm},$$
    as these are already relations of $H_q(3)$. Furthermore, it follows from \eqref{eq:involutive long relation} that
    $$\overline{D} = \sum_{\ell = 1}^3 (-q)^{n - \ell} \qdet(\mathcal{G}_{\hat{\ell} \hat{n}}) g_{\ell n} = \qdet(\mathcal{G}_{\hat{n} \hat{n}}) g_{n n} + (-q)^{n - m} \qdet(\mathcal{G}_{\hat{m} \hat{n}}) g_{m n},$$
    since $\qdet(\mathcal{G}_{\hat{r}\hat{n}}) g_{rn} = 0$ (in case \eqref{item:g_ik = 0} we have $g_{rn} = 0$, and in case \eqref{item:g_kj = 0} we have $\qdet(\mathcal{G}_{\hat{r}\hat{n}}) = 0$). Using that $g_{rn} g_{nr} = g_{rm} g_{mr} = 0$, we get
    $$\qdet(\mathcal{G}_{\hat{n} \hat{n}}) = g_{rr} g_{mm}, \qquad \qdet(\mathcal{G}_{\hat{m} \hat{n}}) = (-q)^{\delta_{r,2}} g_{rr} g_{nm} = \begin{cases}
        g_{rr} g_{nm}, &\text{if } r \in \{1,3\}, \\
        -q g_{22} g_{13}, &\text{if } r = 2, \\
    \end{cases}$$
    and thus
    $$\overline{D} = g_{rr} g_{mm} g_{nn} - q^{-1} g_{rr} g_{nm} g_{mn} = g_{rr}(g_{mm} g_{nn} - q^{-1} g_{nm} g_{mn}) = g_{rr} D',$$
    since $n - m = -1$ when $i \in \{1,3\}$ and $n - m = -2$ when $i = 2$. Now, the relation \eqref{eq:involutive 2x2 minor} gives
    $$\overline{D} = g_{rr} D' = g_{rr}(g_{nn} g_{mm} - q^{-1} g_{nm} g_{mn}) = g_{rr}(g_{mm} g_{nn} - q g_{mn} g_{nm}).$$
    Combining the above equations and using the fact that $g_{rr}$ is invertible, we deduce that
    $$g_{nn} g_{mm} = g_{mm} g_{nn}, \qquad g_{nm} g_{mn} = q^2 g_{mn} g_{nm}.$$
    Once again using \eqref{eq:involutive long relation}, we get
    $$0 = \sum_{\ell = 1}^3 (-q)^{n - \ell} \qdet(\mathcal{G}_{\hat\ell \hat n}) g_{\ell m} = \qdet(\mathcal{G}_{\hat n \hat n}) g_{n m} + (-q)^{n - m} \qdet(\mathcal{G}_{\hat m \hat n}) g_{m m},$$
    since $\qdet(\mathcal{G}_{\hat{r}\hat{n}}) g_{rm} = 0$. Therefore,
    $$g_{rr} g_{m m} g_{n m} - q^{-1} g_{rr} g_{n m} g_{m m} = 0.$$
    Since $g_{rr}$ is invertible, we conclude that $g_{m m} g_{n m} = q^{-1} g_{n m} g_{m m}$. We can similarly prove that $g_{mn} g_{nn} = q^{-1} g_{nn} g_{mn}$ using the relation
    $$\sum_{\ell = 1}^3 (-q)^{m - \ell} \qdet(\mathcal{G}_{\hat\ell \hat m}) g_{\ell n} = 0$$
    from \eqref{eq:involutive long relation}. Therefore, $H'$ satisfies all the relations from \eqref{eq:involutive relations}.

    It remains to check that $H'$ is a Hopf algebra with the correct Hopf structure. There is nothing to check for $\varepsilon$. For the comultiplication, we have
    $$\Delta(g_{nn}) = \sum_{\ell = 1}^3 g_{n\ell} \otimes g_{\ell n} = g_{nn} \otimes g_{nn} + g_{nm} \otimes g_{mn} \in H' \otimes H',$$
    since $g_{nr} \otimes g_{rn} = 0$. Similarly, we also have
    $$\Delta(g_{mm}) = g_{mm} \otimes g_{mm} + g_{mn} \otimes g_{nm} \in H' \otimes H'.$$
    Next, we have
    $$\Delta(g_{nm}) = \sum_{\ell = 1}^3 g_{n\ell} \otimes g_{\ell m} = g_{nn} \otimes g_{nm} + g_{nm} \otimes g_{mm} \in H' \otimes H',$$
    since $g_{nr} \otimes g_{rm} = 0$, and similarly,
    $$\Delta(g_{mn}) = g_{mn} \otimes g_{nn} + g_{mm} \otimes g_{mn} \in H' \otimes H'.$$
    Therefore, $\Delta$ restricts to a well-defined comultiplication on $H'$ which matches up with the comultiplication of $H_2(2)$. It remains to check the antipode: we have
    $$S(g_{nn}) = \qdet(\mathcal{G}_{\hat{n}\hat{n}}) \overline{D}^{-1} = (g_{rr} g_{mm} - q^{\pm 1} g_{rm} g_{mr}) \overline{D}^{-1} = g_{rr} g_{mm} \overline{D}^{-1},$$
    since $g_{rm} g_{mr} = 0$. Using that $\overline{D} = g_{rr} D'$, and that $g_{rr}$ commutes with $D'$ and $g_{mm}$, we conclude that
    $$S(g_{nn}) = g_{mm} (D')^{-1} \in H'.$$
    Proving that $S(g_{mm}) = g_{nn} (D')^{-1} \in H'$ is identical to the above. If $m - n = 1$, then
    $$S(g_{nm}) = -q \qdet(\mathcal{G}_{\hat{m}\hat{n}}) \overline{D}^{-1} = -q (g_{rr} g_{nm} - q^{\pm 1} g_{rm} g_{nr}) \overline{D}^{-1} = -q g_{rr} g_{nm} \overline{D}^{-1},$$
    since $g_{rm} g_{nr} = 0$. If $m - n = 2$ (in which case $n = 1$, $m = 3$, and $r = 2$), then
    $$S(g_{13}) = q^2 \qdet(\mathcal{G}_{\hat{3}\hat{1}}) \overline{D}^{-1} = q^2 (g_{12} g_{23} - q^{-1} g_{13} g_{22}) \overline{D}^{-1} = -q g_{13} g_{22} \overline{D}^{-1},$$
    since $g_{nr} g_{rm} = g_{12} g_{23} = 0$ in this case. In either case, using that $\overline{D} = g_{rr} D'$, we conclude that
    $$S(g_{nm}) = -q g_{nm} (D')^{-1} \in H'.$$
    Checking that $S(g_{mn}) = -q^{-1} g_{mn} (D')^{-1}$ is similar. Therefore, $S$ restricts to an antipode on $H'$, and thus $H'$ is a Hopf subalgebra of $\overline{H}$.

    We can now see that the map $H_q(2) \to H'$ defined by
    $$g_{11} \mapsto g_{nn}, \qquad g_{12} \mapsto g_{nm}, \qquad g_{21} \mapsto g_{mn}, \qquad g_{22} \mapsto g_{mm}$$
    is a surjective homomorphism of Hopf algebras.
\end{proof}

Proposition \ref{prop:Hopf subalgebra} allows us to reduce to $H_q(2)$, so we can use Theorem \ref{thm:cocommutative quotient 2 variables} to get the following consequence.

\begin{cor}\label{cor:Hopf subalgebra}
    Suppose $q \neq \pm 1$. Choose $i,j,k$ such that $\{i,j,k\} = \{1,2,3\}$ and suppose $g_{ij} = 0$, so that $g_{ik} = 0$ or $g_{kj} = 0$ by Lemma \ref{lem:Hopf ideal}. Then the following hold.
    \begin{enumerate}
        \item If $g_{ik} = 0$, then either $g_{jk} = g_{kj} = 0$, or else only one of $g_{jk}$ and $g_{kj}$ is zero and $g_{jj} = g_{kk}$.\label{item:consequences of g_ik = 0}
        \item If $g_{kj} = 0$, then either $g_{ik} = g_{ki} = 0$, or else only one of $g_{ik}$ and $g_{ki}$ is zero and $g_{ii} = g_{kk}$.\label{item:consequences of g_kj = 0}
    \end{enumerate}
    In either case, the diagonal entries $g_{11}, g_{22}, g_{33}$ pairwise commute, and $\overline{D} = g_{11} g_{22} g_{33}$, so $g_{11}, g_{22}, g_{33}$ are units in $\overline{H}$.
\end{cor}
\begin{proof}
    Let $H'$, $m$ and $n$ be as in Proposition \ref{prop:Hopf subalgebra} and its proof, so that $H'$ is generated by $g_{nn}, g_{mm}, g_{nm}, g_{mn}$ and $(D')^{-1}$, where $D' = g_{nn} g_{mm} - q^{-1} g_{nm} g_{mn}$. Since $\overline{H}$ is cocommutative, so is $H'$. Therefore, $H'$ is a cocommutative quotient of $H_q(2)$. By the proof of Theorem \ref{thm:cocommutative quotient 2 variables} and the assumption that $q \neq \pm 1$, it follows that either $H'$ is a quotient of $\kk\ZZ^2$, in which case we have $g_{nm} = g_{mn} = 0$, or $H'$ is a quotient of $A(0,q^{\pm 1})$, in which case one of $g_{nm}$ and $g_{mn}$ is zero and $g_{nn} = g_{mm}$.

    In either case, since $g_{nm} g_{mn} = 0$, we must have $\overline{D} = g_{11} g_{22} g_{33}$ and $g_{11}, g_{22}, g_{33}$ pairwise commute by the relation \eqref{eq:involutive 2x2 minor}.
\end{proof}

\subsubsection{Consequences of cocommutativity}

We now proceed as we did in the two-variable case: we analyze the consequences of requiring $\overline{H}$ to be cocommutative. As one would expect, the conditions we get are more complicated than the ones from Lemma \ref{lem:linear dependence 2 variables}.

\begin{lem}\label{lem:linear dependence 3 variables}
    Choose $i,j,k$ such that $\{i,j,k\} = \{1,2,3\}$. Then the following hold in $\extp^2 \overline{H}$.
    \begin{enumerate}
        \item $g_{12} \wedge g_{21} = g_{31} \wedge g_{13} = g_{23} \wedge g_{32}$. \label{cond:wedge from diagonals}
        \item $(g_{ii} - g_{jj}) \wedge g_{ij} = g_{kj} \wedge g_{ik}$.\label{cond:wedge from ij}
    \end{enumerate}
\end{lem}
\begin{proof}
    In this proof, we make the identification $a \wedge b = \frac{1}{2}(a \otimes b - b \otimes a)$.

    We have
    $$\Delta(g_{11}) = g_{11} \otimes g_{11} + g_{12} \otimes g_{21} + g_{13} \otimes g_{31} = g_{11} \otimes g_{11} + g_{21} \otimes g_{12} + g_{31} \otimes g_{13},$$
    since $\overline{H}$ is cocommutative. Rearranging, we see that
    $$g_{12} \otimes g_{21} - g_{21} \otimes g_{12} = g_{31} \otimes g_{13} - g_{13} \otimes g_{31}.$$
    It follows that $g_{12} \wedge g_{21} = g_{31} \wedge g_{13}$. We can proceed similarly by considering $\Delta(g_{22})$ to conclude that $g_{12} \wedge g_{21} = g_{23} \wedge g_{32}$, which proves \eqref{cond:wedge from diagonals}.

    For \eqref{cond:wedge from ij}, consider
    $$\Delta(g_{12}) = g_{11} \otimes g_{12} + g_{12} \otimes g_{22} + g_{13} \otimes g_{32} = g_{12} \otimes g_{11} + g_{22} \otimes g_{12} + g_{32} \otimes g_{13}.$$
    Therefore, we get
    $$(g_{11} - g_{22}) \otimes g_{12} - g_{12} \otimes (g_{11} - g_{22}) = g_{32} \otimes g_{13} - g_{13} \otimes g_{32}.$$
    It follows that $(g_{11} - g_{22}) \wedge g_{12} = g_{32} \wedge g_{13}$, which proves \eqref{cond:wedge from ij} in the case $(i,j,k) = (1,2,3)$. The other cases of \eqref{cond:wedge from ij} follow similarly.
\end{proof}

The next result is a well-known lemma which characterizes when $a \wedge b = c \wedge d$, where $a,b,c,d$ are elements of a vector space.

\begin{lem}\label{lem:wedge}
    Let $V$ be a vector space and suppose $a \wedge b = c \wedge d$ in $\extp^2 V$, where $a,b,c,d \in V$. Then either $a \wedge b = c \wedge d = 0$, or there exists a matrix $M \in SL_2(\kk)$ such that
    $$\begin{pmatrix}
        c \\
        d
    \end{pmatrix} = M \begin{pmatrix}
        a \\
        b
    \end{pmatrix}.$$
\end{lem}
\begin{proof}
    Suppose $a \wedge b = c \wedge d \neq 0$. Then it must be the case that $\spn\{a,b,c,d\}$ is two-dimensional, with basis given by $\{a,b\}$ or $\{c,d\}$. In particular, there exist $\alpha, \beta, \gamma, \delta \in \kk$ such that
    $$c = \alpha a + \beta b, \quad d = \gamma a + \delta b.$$
    It follows that
    \begin{align*}
        a \wedge b &= c \wedge d = (\alpha a + \beta b) \wedge (\gamma a + \delta b) \\
        &= \alpha \delta (a \wedge b) + \beta \gamma (b \wedge a) \\
        &= (\alpha \delta - \beta \gamma) a \wedge b.
    \end{align*}
    Therefore, $\alpha \delta - \beta \gamma = 1$, and thus the result follows by setting $M \coloneqq \begin{pmatrix}
        \alpha & \beta \\
        \gamma & \delta
    \end{pmatrix}$.
\end{proof}

\subsubsection{The case \texorpdfstring{$q \neq \pm 1$}{where q is not 1 or -1}}\label{subsec:q^4 neq 1}

We now start working toward a proof of Theorem \ref{thm:generic}. Our first aim is to show that at least one of the generators $g_{ij}$ is zero, which will allow us to use Corollary \ref{cor:Hopf subalgebra}.

In the following result, we show that all the wedge products from Lemma \ref{lem:linear dependence 3 variables}\eqref{cond:wedge from diagonals} are zero using Lemma \ref{lem:wedge}. The proof is immediate from the relations \eqref{eq:involutive relations 3 variables}, unless $q^6 = 1$.

\begin{lem}\label{lem:opposite wedges are zero}
    Suppose $q \neq \pm 1$. Then
    $$g_{12} \wedge g_{21} = g_{31} \wedge g_{13} = g_{23} \wedge g_{32} = 0.$$
\end{lem}
\begin{proof}
    Assume, for a contradiction, that $g_{12} \wedge g_{21} = g_{31} \wedge g_{13} = g_{23} \wedge g_{32} \neq 0$ (these are all equal by Lemma \ref{lem:linear dependence 3 variables}\eqref{cond:wedge from diagonals}). Then Lemma \ref{lem:wedge} implies that there exists $M = \begin{pmatrix}
        \alpha & \beta \\
        \gamma & \delta
    \end{pmatrix} \in SL_2(\kk)$ such that $\begin{pmatrix}
        g_{31} \\
        g_{13}
    \end{pmatrix} = M \begin{pmatrix}
        g_{12} \\
        g_{21}
    \end{pmatrix}$. In other words,
    $$g_{31} = \alpha g_{12} + \beta g_{21}, \quad g_{13} = \gamma g_{12} + \delta g_{21}.$$
    Right-multiplying the second equation by $\overline{D}$, we get $g_{13}\overline{D} = \gamma g_{12}\overline{D} + \delta g_{21}\overline{D}$, and therefore
    $$q^4 \overline{D} g_{13} = \gamma q^2 \overline{D} g_{12} + \delta q^{-2} \overline{D} g_{21},$$
    by \eqref{eq:involutive qdet commutation}. Left-multiplying by $q^{-4}\overline{D}^{-1}$, it follows that $g_{13} = \gamma q^{-2} g_{12} + \delta q^{-6} g_{21}$. Thus,
    $$\gamma g_{12} + \delta g_{21} = g_{13} = \gamma q^{-2} g_{12} + \delta q^{-6} g_{21}.$$
    Now, $g_{12}$ and $g_{21}$ are linearly independent by assumption, and thus $\gamma = \gamma q^{-2}$ and $\delta = \delta q^{-6}$. Since we are assuming that $q \neq \pm 1$, it follows that $\gamma = 0$. Furthermore, if $q^6 \neq 1$, then $\delta = 0$, which would contradict $M \in SL_2(\kk)$. Therefore, it must be the case that $q^6 = 1$. Since $M \in SL_2(\kk)$, we have $\alpha \neq 0$ and $\delta = \alpha^{-1}$, so we get $g_{13} = \alpha^{-1} g_{21}$.  By right multiplying the other equation $g_{31} = \alpha g_{12} + \beta g_{21}$ by $\overline{D}$ and arguing similarly, we can that $\beta = 0$, and thus $g_{31} = \alpha g_{12}$.

    We also have $g_{23} \wedge g_{32} = g_{31} \wedge g_{13} \neq 0$, so proceeding as above, there exists another matrix $M' = \begin{pmatrix}
        \lambda & \mu \\
        \nu & \eta
    \end{pmatrix} \in SL_2(\kk)$ such that $\begin{pmatrix}
        g_{31} \\
        g_{13}
    \end{pmatrix} = M' \begin{pmatrix}
        g_{23} \\
        g_{32}
    \end{pmatrix}$. In other words,
    $$g_{31} = \lambda g_{23} + \mu g_{32}, \quad g_{13} = \nu g_{23} + \eta g_{32}.$$
    Similarly as in the previous paragraph, we deduce that $\mu = \nu = 0$ and that $\eta = \lambda^{-1}$, and thus $g_{31} = \lambda g_{23}$ and $g_{13} = \lambda^{-1} g_{32}$. Combining the above, we have
    \begin{equation}\label{eq:q^6 = 1 relations between g_ij}
        g_{31} = \alpha g_{12} = \lambda g_{23}, \quad g_{13} = \alpha^{-1} g_{21} = \lambda^{-1} g_{32}.
    \end{equation}
    Now, Lemma \ref{lem:linear dependence 3 variables}\eqref{cond:wedge from ij} implies that $(g_{11} - g_{22}) \wedge g_{12} = g_{32} \wedge g_{13}$. Combining this with \eqref{eq:q^6 = 1 relations between g_ij}, we get
    $$(g_{11} - g_{22}) \wedge g_{12} = g_{32} \wedge g_{13} = g_{32} \wedge (\lambda^{-1} g_{32}) = 0,$$
    so $g_{12}$ and $g_{11} - g_{22}$ are linearly dependent, say $g_{11} - g_{22} = \xi g_{12}$ for some $\xi \in \kk$. Note that \eqref{eq:involutive qdet commutation} implies that
    $$(g_{11} - g_{22}) \overline{D} = \xi g_{12} \overline{D} = q^2 \xi \overline{D} g_{12} = q^2 \overline{D} (g_{11} - g_{22}) = q^2 (g_{11} - g_{22}) \overline{D}.$$
    Right-multiplying by $\overline{D}^{-1}$, it follows that $g_{11} - g_{22} = q^2 (g_{11} - g_{22})$. Since $q^2 \neq 1$ by assumption, we conclude that $g_{11} = g_{22}$.
    
    We can proceed similarly using Lemma \ref{lem:linear dependence 3 variables} to further deduce that $g_{11} = g_{22} = g_{33}$. Therefore, the above shows that $\overline{H}$ is generated by the three elements $g_{11}$, $g_{12}$, and $g_{21}$.
    
    By \eqref{eq:involutive q commute}, we have $g_{12} g_{11} = q g_{11} g_{12}$. Similarly, \eqref{eq:involutive q commute} also gives $g_{33} g_{31} = q g_{31} g_{33}$, which implies that
    $$g_{11} g_{12} = q g_{12} g_{11},$$
    where we used that $g_{33} = g_{11}$ and that $g_{31} = \alpha g_{12}$ from \eqref{eq:q^6 = 1 relations between g_ij}. But then
    $$g_{12} g_{11} = q g_{11} g_{12} = q^2 g_{12} g_{11},$$
    so $g_{12} g_{11} = 0$, since $q^2 \neq 1$. Similarly, we can also deduce that
    \begin{equation}\label{eq:q^6 = 1 quadratic relations}
        g_{11} g_{12} = g_{12} g_{11} = g_{11} g_{21} = g_{21} g_{11} = g_{12} g_{21} = g_{21} g_{12} = 0.
    \end{equation}
    Now, we can use \eqref{eq:involutive 2x2 minor} to give
    $$g_{11} g_{32} - q g_{12} g_{31} = g_{32} g_{11} - q^{-1} g_{31} g_{12}.$$
    Using that $g_{32} = \lambda \alpha^{-1} g_{21}$ and that $g_{31} = \alpha g_{12}$ from \eqref{eq:q^6 = 1 relations between g_ij}, we deduce that
    $$\lambda \alpha^{-1} g_{11} g_{21} - q \alpha g_{12}^2 = \lambda \alpha^{-1} g_{21} g_{11} - q^{-1} \alpha g_{12}^2.$$
    But we know from \eqref{eq:q^6 = 1 quadratic relations} that $g_{11} g_{21} = g_{21} g_{11} = 0$, so we conclude that $q g_{12}^2 = q^{-1} g_{12}^2$. Since $q^2 \neq 1$, it follows that $g_{12}^2 = 0$. By a completely symmetric argument, we can also deduce that $g_{21}^2 = 0$.

    Recall that
    $$\overline{D} = g_{11}g_{22}g_{33} - q^{-1} g_{11}g_{23}g_{32} - q^{-1} g_{12}g_{21}g_{33} + q^{-2} g_{12}g_{23}g_{31} + q^{-2} g_{13}g_{21}g_{32} - q^{-3} g_{13}g_{22}g_{31}.$$
    Using equations \eqref{eq:q^6 = 1 relations between g_ij} and \eqref{eq:q^6 = 1 quadratic relations}, as well as the fact that $g_{12}^2 = g_{21}^2 = 0$, we see that $\overline{D} = g_{11}^3$. But $\overline{D}$ is invertible in $\overline{H}$, and thus $g_{11}$ must also be invertible. Now, \eqref{eq:q^6 = 1 quadratic relations} implies that $g_{12} = g_{21} = 0$, a contradiction.
\end{proof}

Next, we show that some more of the wedge products from Lemma \ref{lem:linear dependence 3 variables} are zero.

\begin{lem}\label{lem:easy cocommutative consequences}
    Suppose $q \neq \pm 1$. Then
    $$g_{21} \wedge g_{32} = g_{23} \wedge g_{12} = (g_{11} - g_{33}) \wedge g_{13} = (g_{11} - g_{33}) \wedge g_{31} = 0.$$
\end{lem}
\begin{proof}
    Fix $i,j \in \{1,3\}$ with $i \neq j$. By Lemma \ref{lem:linear dependence 3 variables}, we have
    $$(g_{ii} - g_{jj}) \wedge g_{ij} = g_{2,j} \wedge g_{i,2}.$$
    Assume, for a contradiction, that $(g_{ii} - g_{jj}) \wedge g_{ij} = g_{2,j} \wedge g_{i,2} \neq 0$. In particular, this means that $g_{ii} \neq g_{jj}$ and that $g_{ij}$, $g_{2,j}$, and $g_{i,2}$ are all nonzero. By Lemma \ref{lem:wedge}, there exists $M = \begin{pmatrix}
        \alpha & \beta \\
        \gamma & \delta
    \end{pmatrix} \in SL_2(\kk)$ such that
    \begin{equation}\label{eq:SL2 matrix}
        g_{ii} - g_{jj} = \alpha g_{2,j} + \beta g_{i,2}, \quad g_{ij} = \gamma g_{2,j} + \delta g_{i,2}.
    \end{equation}
    Right-multiplying the first of equations in \eqref{eq:SL2 matrix} by $\overline{D}$, we get
    $$(g_{ii} - g_{jj})\overline{D} = \alpha g_{2,j}\overline{D} + \beta g_{i,2} \overline{D}.$$
    By \eqref{eq:involutive qdet commutation}, it follows that
    $$\overline{D}(g_{ii} - g_{jj}) = \alpha q^{2(j - 2)} \overline{D} g_{2,j} + \beta q^{2(2 - i)} \overline{D} g_{i,2}.$$
    Left-multiplying by $\overline{D}^{-1}$, we deduce that
    $$g_{ii} - g_{jj} = \alpha q^{2(j - 2)} g_{2,j} + \beta q^{2(2 - i)} g_{i,2}.$$
    Using that $g_{ii} - g_{jj} = \alpha g_{2,j} + \beta g_{i,2}$, we get
    $$\alpha g_{2,j} + \beta g_{i,2} = \alpha q^{2(j - 2)} g_{2,j} + \beta q^{2(2 - i)} g_{i,2}.$$
    Since $g_{2,j}$ and $g_{i,2}$ are linearly independent, we conclude that $\alpha = \alpha q^{2(j - 2)}$ and $\beta = \beta q^{2(2 - i)}$. But $i,j \in \{1,3\}$, so it follows that $2 - i = j - 2 = \pm 1$. In other words, we have $\alpha = q^{\pm 2} \alpha$ and $\beta = q^{\pm 2} \beta$. By the assumption that $q^2 \neq 1$, we must have $\alpha = \beta = 0$. This contradicts $M \in SL_2(\kk)$.
\end{proof}

We are now able to show that at least one of the generators $g_{ij}$ is zero in $\overline{H}$.

\begin{lem}\label{lem:one of the g_ij is zero}
    Suppose $q \neq \pm 1$. Then $g_{ij} = 0$ for some $i \neq j$.
\end{lem}
\begin{proof}
    Assume, for a contradiction, that $g_{ij} \neq 0$ for all $i,j$. By Lemma \ref{lem:opposite wedges are zero}, we have
    $$g_{21} = \lambda g_{12}, \qquad g_{31} = \mu g_{13},$$
    for some $\lambda, \mu \in \kk^*$. Furthermore, Lemma \ref{lem:easy cocommutative consequences} implies that
    $$g_{23} = \alpha g_{12}, \qquad g_{32} = \beta g_{12}, \qquad g_{33} = g_{11} + \gamma g_{13},$$
    for some $\alpha, \beta, \gamma \in \kk$ with $\alpha, \beta \neq 0$.
    
    Now, it follows by Lemma \ref{lem:linear dependence 3 variables} that
    $$(g_{11} - g_{22}) \wedge g_{12} = g_{32} \wedge g_{13} = \beta g_{12} \wedge g_{13}.$$
    Assume, for a contradiction, that $g_{12} \wedge g_{13} = 0$, meaning $g_{13} = \xi g_{12}$ for some $\xi \in \kk^*$. Then
    $$g_{13} \overline{D} = \xi g_{12} \overline{D} = \xi q^2 \overline{D} g_{12} = q^2 \overline{D} g_{13} = q^{-2} g_{13} \overline{D},$$
    where we used \eqref{eq:involutive qdet commutation}. This implies that $g_{13} = 0$, a contradiction. Therefore, $g_{12} \wedge g_{13} \neq 0$.

    Since we have
    $$(g_{11} - g_{22}) \wedge g_{12} = \beta g_{12} \wedge g_{13} \neq 0,$$
    Lemma \ref{lem:wedge} implies that there exists a matrix $M = \begin{pmatrix}
        a & b \\
        c & d
    \end{pmatrix} \in SL_2(\kk)$ such that
    $$g_{11} - g_{22} = a g_{12} + b\beta g_{13}, \qquad g_{12} = c g_{12} + d\beta g_{13}.$$
    The linear independence of $g_{12}$ and $g_{13}$ implies that $c = 1$ and $d = 0$. Since $M \in SL_2(\kk)$, it follows that $b = -1$. Therefore,
    $$g_{11} - g_{22} = a g_{12} - \beta g_{13}.$$
    Multiplying both sides by $\overline{D}$. we have
    $$(g_{11} - g_{22}) \overline{D} = (a g_{12} - \beta g_{13})\overline{D} = \overline{D} (a q^2 g_{12} - \beta q^4 g_{13}).$$
    Since $\overline{D}$ commutes with $g_{11}$ and $g_{22}$, we conclude that
    $$g_{11} - g_{22} = a q^2 g_{12} - \beta q^4 g_{13}$$
    upon multiplying by $\overline{D}^{-1}$. Using that $g_{11} - g_{22} = a g_{12} - \beta g_{13}$, we get
    $$a g_{12} - \beta g_{13} = a q^2 g_{12} - \beta q^4 g_{13}.$$
    It follows that $a = a q^2$ and that $\beta = \beta q^4$. However, since $q^2 \neq 1$ by assumption, we conclude that $a = 0$. Therefore, we have $g_{22} = g_{11} + \beta g_{13}$.

    Summarizing, we have
    $$\begin{pmatrix}
        g_{11} & g_{12} & g_{13} \\
        g_{21} & g_{22} & g_{23} \\
        g_{31} & g_{32} & g_{33}
    \end{pmatrix} = \begin{pmatrix}
        g_{11} & g_{12} & g_{13} \\
        \lambda g_{12} & g_{11} + \beta g_{13} & \alpha g_{12} \\
        \mu g_{13} & \beta g_{12} & g_{11} + \gamma g_{13}
    \end{pmatrix}.$$
    We now analyze what the relations \eqref{eq:involutive relations 3 variables} become under the above conditions:
    \begin{itemize}
        \item $g_{23} g_{21} = q g_{21} g_{23}$ implies that $g_{12}^2 = 0$, where we used that $q \neq 1$.
        \item $g_{11} g_{22} - q^{-1} g_{12} g_{21} = g_{22} g_{11} - q g_{21} g_{12}$ implies that $g_{11} g_{13} = g_{13} g_{11}$, where we used that $g_{12}^2 = 0$.
        \item $g_{13} g_{11} = q g_{11} g_{13}$ implies that $g_{11} g_{13} = g_{13} g_{11} = 0$, where we used that $g_{11} g_{13} = g_{13} g_{11}$ and that $q \neq 1$.
        \item $g_{11} g_{33} - q^{-1} g_{13} g_{31} = g_{33} g_{11} - q g_{31} g_{13}$ implies that $g_{13}^2 = 0$, where we used that $g_{11} g_{13} = g_{13} g_{11} = 0$ and that $q^2 \neq 1$.
    \end{itemize}
    It now follows that $\overline{D} = g_{11}^3$, which implies that $g_{11}$ is invertible. But then it follows from the equation $g_{11} g_{13} = 0$ that $g_{13} = 0$, a contradiction.
\end{proof}

Now that we have shown that some $g_{ij}$ must be zero, we will use this to show that many more entries in the matrix $\mathcal{G}$ must be $0$.  Next, we show that at least one of the opposite generators $g_{ij}$ and $g_{ji}$ must be zero for every $i \neq j$. This shows that one of the terms in each of the wedge products from Lemma \ref{lem:linear dependence 3 variables}\eqref{cond:wedge from diagonals} are zero.

\begin{lem}\label{lem:opposite entries are zero}
    Suppose $q \neq \pm 1$ and let $i,j \in \{1,2,3\}$ with $i \neq j$. Then $g_{ij} = 0$ or $g_{ji} = 0$.
\end{lem}
\begin{proof}
    Assume, for a contradiction, that $g_{ij}$ and $g_{ji}$ are both nonzero, for some $i \neq j$. By Lemma \ref{lem:one of the g_ij is zero}, it follows that $g_{k\ell} = 0$ for some $k \neq \ell$. Then Lemma \ref{lem:Hopf ideal} implies that $g_{km} = 0$ or $g_{m\ell} = 0$, where $m \in \{1,2,3\} \setminus \{k,\ell\}$. It must be the case that either $k$ or $\ell$ (but not both) is equal to one of $i$ or $j$, since there are only three possibilities for each of these elements. Without loss of generality, we may assume that $k$ or $\ell$ is equal to $i$ (and thus $m = j$).

    \begin{case}
        $k = i$.
    \end{case}

    In this case, the above gives that $g_{j\ell} = 0$, so it follows from Lemma \ref{lem:Hopf ideal} that $g_{i\ell} = 0$. But now Corollary \ref{cor:Hopf subalgebra} implies that $g_{ij}$ or $g_{ji}$ is zero, a contradiction.

    \begin{case}
        $\ell = i$.
    \end{case}

    In this case, we have $g_{kj} = 0$, and thus $g_{ki} = 0$, by Lemma \ref{lem:Hopf ideal}. But now Corollary \ref{cor:Hopf subalgebra} implies that $g_{ij}$ or $g_{ji}$ is zero, a contradiction.
\end{proof}

In the following result, we show that if $g_{ij} \neq 0$, then the diagonal generators $g_{ii}$ and $g_{jj}$ are equal.

\begin{lem}\label{lem:g_ij nonzero implies g_ii = g_jj}
    Suppose $q \neq \pm 1$ and let $i,j \in \{1,2,3\}$ with $i \neq j$. If $g_{ij} \neq 0$, then $g_{ii} = g_{jj}$.
\end{lem}
\begin{proof}
    Suppose $g_{ij} \neq 0$, so $g_{ji} = 0$ by Lemma \ref{lem:opposite entries are zero}. Let $k$ be the unique element of $\{1,2,3\} \setminus \{i,j\}$.  

    \begin{case}\label{case:g_ki = g_kj = 0}
        $g_{ki} = g_{kj} = 0$.   
    \end{case}

    In this case we can apply Corollary~\ref{cor:Hopf subalgebra}\eqref{item:consequences of g_ik = 0} to conclude that $g_{ii} = g_{jj}$, since $g_{ij} \neq 0$.

    \begin{case}\label{case:g_ik = g_jk = 0}
        $g_{ik} = g_{jk} = 0$.
    \end{case}

    Similar to the previous case, except we apply Corollary~\ref{cor:Hopf subalgebra}\eqref{item:consequences of g_kj = 0}.

    \begin{case}
       Neither of the previous two cases holds.
    \end{case}
    
    Note that since $g_{ji} = 0$, it follows from Lemma \ref{lem:Hopf ideal} that either $g_{ki} = 0$ or $g_{jk} = 0$.  If $g_{ki} = 0$ then we must have $g_{kj} \neq 0$; else we are in Case \ref{case:g_ki = g_kj = 0}. Then $g_{kj} \neq 0$ forces $g_{jk} = 0$ by Lemma~\ref{lem:opposite entries are zero}. Now we can assume that $g_{ik} \neq 0$; else 
    we are in Case \ref{case:g_ik = g_jk = 0}.  Thus we have $g_{ji} = g_{ki} = g_{jk} = 0$ while $g_{kj} \neq 0$, $g_{ik} \neq 0$.
    If we start instead with the assumption that $g_{jk} = 0$, it is easy to check that we reach the same conclusion.

    Now using $g_{ji} = g_{jk} = 0$ and $g_{ik} \neq 0$, Corollary~\ref{cor:Hopf subalgebra}\eqref{item:consequences of g_ik = 0} implies $g_{ii} = g_{kk}$.  Similarly, using $g_{ji} = g_{ki} = 0$ and $g_{kj} \neq 0$, Corollary \ref{cor:Hopf subalgebra}\eqref{item:consequences of g_kj = 0} implies $g_{jj} = g_{kk}$.  Therefore, $g_{ii} = g_{jj}$ in this case as well.  
\end{proof}

When some of the diagonal elements $g_{ii}$ are equal, it helps us to show additional $g_{jk}$'s are zero.

\begin{lem}\label{lem:leveraging equal diagonal elements}
    Suppose $q \neq \pm 1$ and choose $i,j,k$ such that $\{i,j,k\} = \{1,2,3\}$.  If $g_{ii} = g_{jj}$ and either $g_{ji} = 0$ or $g_{ik} = 0$, then $g_{jk} = 0$.
\end{lem}
\begin{proof}
    By Lemma~\ref{lem:easy calculation}, we have $g_{ii} g_{jk} = q^{\varepsilon} g_{jk} g_{ii}$ for the value of 
    $\varepsilon$ calculated there.  In addition, we have $g_{jj} g_{jk} = q^{\beta} g_{jk} g_{jj}$ by relations \eqref{eq:involutive q commute}, where $\beta = 1$ if $j > k$ and $\beta = -1$ if $k > j$.  Now since $g_{ii} = g_{jj}$, combining the two relations we have $(q^{\varepsilon} - q^{\beta}) g_{jk} g_{ii} = 0$.

    Now if $i \in \{1, 3 \}$ we have $\varepsilon = 0$, $\beta \in \{1, -1 \}$; if $(i,j,k) = (2,1,3)$ we have $\varepsilon = -2$, $\beta = -1$; and if $(i,j,k) = (2,3,1)$ we have $\varepsilon = 2$, $\beta = 1$.  In all cases 
    we conclude that $(q-1) g_{jk} g_{ii} = 0$ and since $q \neq 1$, then $g_{jk} g_{ii} = 0$.  Finally, $g_{ii}$ is a unit by Corollary~\ref{cor:Hopf subalgebra}, and so $g_{jk} = 0$.
\end{proof}

We now finish showing that at least one of the terms in each of the wedges appearing in Lemma \ref{lem:linear dependence 3 variables} is zero.

\begin{lem}\label{lem:g_ik or g_kj is zero}
    Suppose $q \neq \pm 1$ and choose $i,j,k$ such that $\{i,j,k\} = \{1,2,3\}$. Then $g_{ik} = 0$ or $g_{kj} = 0$.
\end{lem}
\begin{proof}
    Assume, for a contradiction, that $g_{ik}$ and $g_{kj}$ are both nonzero. Lemma \ref{lem:opposite entries are zero} implies that $g_{ki} = g_{jk} = 0$, while Lemma \ref{lem:g_ij nonzero implies g_ii = g_jj} implies that $g_{11} = g_{22} = g_{33}$. Now by Lemma~\ref{lem:leveraging equal diagonal elements}, $g_{ii} = g_{jj}$ together with 
    $g_{jk} = 0$ imply $g_{ik} = 0$, a contradiction.
\end{proof}

We summarize the above results in the following corollary.

\begin{cor}\label{cor:many dichotomies}
    Suppose $q \neq \pm 1$ and choose $i,j,k$ such that $\{i,j,k\} = \{1,2,3\}$. Then the following hold in $\overline{H}$.
    \begin{enumerate}
        \item $g_{ij} = 0$ or $g_{ji} = 0$.\label{cond:ij or ji is zero}
        \item $g_{ii} = g_{jj}$ or $g_{ij} = 0$.\label{cond:ii = jj or ij is zero}
        \item $g_{ik} = 0$ or $g_{kj} = 0$.\label{cond:ik or kj is zero}
    \end{enumerate}
\end{cor}
\begin{proof}
    Follows from Lemmas \ref{lem:opposite entries are zero}, \ref{lem:g_ij nonzero implies g_ii = g_jj}, and \ref{lem:g_ik or g_kj is zero}.
\end{proof}

The next result is the last ingredient we need to prove Theorem \ref{thm:generic}.

\begin{prop}\label{prop:off-diagonals are mostly zero}
    Suppose $q \neq \pm 1$. If $g_{ij} \neq 0$ for some $i \neq j$, then $g_{ii} = g_{jj}$, and $g_{k\ell} = 0$ for all $k \neq \ell$ such that $(k,\ell) \neq (i,j)$.
\end{prop}
\begin{proof}
    The assumption $g_{ij} \neq 0$ immediately implies $g_{ii} = g_{jj}$ and $g_{ji} = g_{jk} = g_{ki} = 0$ by Corollary~\ref{cor:many dichotomies}. Thus, to prove the result in this case, we must show that $g_{ik} = g_{kj} = 0$.
    
    Corollary~\ref{cor:many dichotomies} also shows that we cannot have both $g_{ik}$ and $g_{kj}$ nonzero.  Suppose 
    that $g_{ik} = 0$ but $g_{kj} \neq 0$.  This forces $g_{jj} = g_{kk}$ as well, so $g_{ii} = g_{jj} = g_{kk}$.
    Now Lemma~\ref{lem:leveraging equal diagonal elements} shows that because $g_{ii} = g_{kk}$ and $g_{ik} = 0$, then $g_{ij} = 0$, which is a contradiction.

    Similarly, if $g_{ik} \neq 0$ but $g_{kj} = 0$, we get $g_{ii} = g_{kk}$. Now by Lemma~\ref{lem:leveraging equal diagonal elements}, $g_{ii} = g_{kk}$ and $g_{kj} = 0$ together imply the contradiction $g_{ij} = 0$ again.
\end{proof}

We are now ready to prove Theorem \ref{thm:generic}.

\begin{proof}[Proof of Theorem \ref{thm:generic}]
    Let $C$ be a cocommutative Hopf algebra, and let $\varphi \colon H_q(3) \to C$ be a Hopf algebra homomorphism. Write $g_{ij} \coloneqq \varphi(h_{ij})$. If $g_{ij} = 0$ for all $i \neq j$, then it is easy to see that $\overline{H} \coloneqq \im(\varphi)$ is a quotient of $\kk \ZZ^3$. So, assume that $g_{ij} \neq 0$ for some $i \neq j$. By Proposition \ref{prop:off-diagonals are mostly zero}, it must be the case that $g_{ii} = g_{jj}$ and $g_{k\ell} = 0$ for all $k \neq \ell$ such that $(k,\ell) \neq (i,j)$. It is now clear from Proposition \ref{prop:generic quotient Bq or Cq} that any Hopf map $H_q(3) \to C$, where $C$ is cocommutative, factors through one of the algebras on our list.  From the universal property of $H_q(3)$, it is clear that any coaction on $A_q(3)$ by a cocommutative Hopf algebra factors through one of the given ones.  Finally, it is straightforward to check that the largest group algebra occurring as a Hopf algebra quotient of $\mathcal{B}_{q^{\pm 1}}$ or $\mathcal{C}_{q^{\pm 1}}$ is $\kk \ZZ^2$, so every coacting group algebra is abelian.
\end{proof}

\section{Applications and questions}\label{sec:applications}

We now spend some time presenting some applications of the main results of the paper, as well as some open questions.

\subsection{GK-dimension 1 Hopf coactions on \texorpdfstring{$A_q(2)$}{Aq(2)}}

First, we partially answer the following question of Chan--Walton--Zhang.

\begin{qstn}[{\cite[Question 5.11]{ChanWaltonZhang}}] \label{qstn:noncommutative coactions on Aq(2)}
    Let $q \in \kk$ such that $q$ is not a root of unity. Is there a noncommutative Hopf algebra of GK-dimension 1 which coacts on $A_q(2)$ inner-faithfully?
\end{qstn}

In \cite[Theorem 0.4]{ChanWaltonZhang}, it is shown that if a finite-dimensional Hopf algebra $H$ acts inner-faithfully on $A_q(2)$, where $q$ is not a root of unity, then $H$ is a group algebra. In fact, by examining the proof of \cite[Theorem 0.4]{ChanWaltonZhang}, it is easy to see that $H$ must be the group algebra of a finite abelian group, and that the same result also holds for coactions by finite-dimensional Hopf algebras. Therefore, Chan--Walton--Zhang proved that all finite-dimensional (in other words, GK-dimension 0) inner-faithful coactions on $A_q(2)$ for $q$ not a root of unity are by group algebras of finite abelian groups.

Using Theorem \ref{thm:cocommutative quotient 2 variables}, we know that there are inner-faithful coactions on $A_q(2)$ by noncommutative Hopf algebras of GK-dimension $2$ -- these are the Hopf algebras $\mathcal{A}(0,q^{\pm 1})$ from Definition \ref{dfn:Goodearl-Zhang Hopf algebra}. If $q$ is a root of unity, then it is easy to show that $\mathcal{A}(0,q^{\pm 1})$ has a noncommutative quotient of GK-dimension $1$.

\begin{cor}
    Suppose $q \neq 1$ and $q^n = 1$ for some $n \geq 2$. Then the noncommutative Hopf algebra
    $$\frac{\mathcal{A}(0,q^{\pm 1})}{(x^n - 1)} \cong \frac{\kk\ang{x,y}}{(x^n - 1, yx - q^{\pm 1} xy)}$$
    has GK-dimension 1 and coacts on $A_q(2)$ inner-faithfully. \qed
\end{cor}

However, if $q$ is not a root of unity, then there are no noncommutative, cocommutative Hopf algebras of GK-dimension $1$ which coact on $A_q(2)$ inner-faithfully, as we show next.

\begin{prop}\label{prop:GKdim 1 cocommutative coactions}
    Suppose $q$ is not a root of unity. Then there are no noncommutative, cocommutative Hopf algebras of GK-dimension $1$ which coact on $A_q(2)$ inner-faithfully.
\end{prop}
\begin{proof}
    Assume, for a contradiction, that $H$ is a noncommutative, cocommutative Hopf algebra of GK-dimension $1$ which coacts on $A_q(2)$ inner-faithfully. By Theorem \ref{thm:cocommutative quotient 2 variables}, it must be the case that $H \cong \mathcal{A}(0,q^{\pm 1})/I$ for some Hopf ideal $I$ of $\mathcal{A}(0,q^{\pm 1})$. Let $\overline{x}$ and $\overline{y}$ be the images of $x$ and $y$ in $H$.
    
    Since $H$ is cocommutative, it follows by the Cartier--Kostant--Gabriel theorem \cite[Theorem 5.10.2]{EGNO} that $H \cong U(P(H)) \# \kk G(H)$, where $P(H) = \kk \overline{y}$ is the Lie algebra of primitive elements of $H$ and $G(H) = \ang{\overline{x}}$ is the group of group-like elements of $H$. It is easy to see that $\overline{y} \neq 0$, as otherwise $H \cong \mathcal{A}(0,q^{\pm 1})/I$ would be commutative. Therefore, $P(H)$ is a one-dimensional Lie algbera, so $\GKdim(U(P(H))) = \GKdim(\kk[\overline{y}]) = 1$.

    By \cite[Lemma 2.2]{AndruskiewitschAngionoHeckenberger}, we see that
    $$1 = \GKdim(H) = \GKdim(U(P(H))) + \GKdim(\kk G(H)) = 1 + \GKdim(G(H)).$$
    Therefore, $\GKdim(\kk G(H)) = 0$, in other words, $G(H)$ is a finite group. It follows that $G(H) = \ang{\overline{x}}$ is a finite cyclic group, so $\overline{x}^n = 1$ for some $n \in \NN$. Hence, we have
    $$\overline{y} = \overline{y} \, \overline{x}^n = q^{\pm n} \overline{x}^n \overline{y} = q^{\pm n} \overline{y}.$$
    However, $q$ is not a root of unity, so $q^n \neq 1$, and thus $\overline{y} = 0$, a contradiction.
\end{proof}

Proposition \ref{prop:GKdim 1 cocommutative coactions} suggests that the answer to Question \ref{qstn:noncommutative coactions on Aq(2)} may be ``no'', but our methods do not yield any information about coactions by non-cocommutative Hopf algebras which may coact on $A_q(2)$.

\subsection{Cosemisimple Hopf coactions on commutative domains}

Next, we consider the following question of Julien Bichon.

\begin{qstn}[Julien Bichon, cf. {\cite[Question 2.21]{WaltonWang}}] \label{qstn:Bichon cosemisimple}
    Is there an infinite-dimensional, noncommutative, cosemisimple, involutive Hopf algebra which coacts on a commutative domain inner-faithfully?
\end{qstn}

Question \ref{qstn:Bichon cosemisimple} is related to \cite[Theorem 4.1]{EtingofWalton}, which shows that if $H$ is a finite-dimensional cosemisimple Hopf algebra which coacts on a commutative domain inner-faithfully, then $H$ is the group algebra of a finite abelian group. Note that, by the Radford--Larson theorem \cite[Theorem 16.1.2]{Radford}, if $H$ is a finite-dimensional Hopf algebra then the following conditions are equivalent:
\begin{itemize}
    \item $H$ is semisimple.
    \item $H$ is cosemisimple.
    \item $H$ is involutive.
\end{itemize}
However, when $H$ is infinite-dimensional, the above are no longer equivalent. Therefore, an analog of \cite[Theorem 4.1]{EtingofWalton} for infinite-dimensional Hopf algebras would necessarily involve choosing some subset of the three conditions above. One such possibility is Question \ref{qstn:Bichon cosemisimple}, which to our knowledge is still an open question.

From Corollary~\ref{cor:involutive relations 2 variables}, we see that the universal involutive Hopf algebra $H_1(2)$ right coacting on $\kk[x,y]$ is commutative (in fact $H_1(2) \cong \mc{O}_{GL_2}$, the coordinate ring of the general linear group).  Thus any involutive Hopf algebra which coacts on $\kk[x,y]$ inner-faithfully must be commutative, as was also noted in \cite[Proposition 5.4]{ChanWaltonZhang}.

Therefore, one natural starting point for Question \ref{qstn:Bichon cosemisimple} is to consider the following question: are there noncommutative involutive Hopf algebras which coact on $\kk[x,y,z]$ inner-faithfully? To answer this, we can study the Hopf algebra $H_1(3)$. However, it is not clear from the presentation given in \eqref{eq:involutive relations} whether $H_1(3)$ is commutative, just like $H_1(2)$ turned out to be commutative. As we show next, the Hopf algebra $H_1(3)$ is noncommutative and has significantly worse properties than $H_1(2)$.

\begin{prop}\label{prop:involutive Hopf algebra is not commutative}
    The Hopf algebra $H_1(3)$ is noncommutative, non-noetherian, and has infinite GK-dimension. In particular, $H_1(3) \not\cong \mc{O}_{GL_3}$.
\end{prop}
\begin{proof}
    Let $I \coloneqq (h_{11} - 1, h_{22} - 1, h_{33} - 1, h_{12}, h_{21}, h_{31}, h_{32})$ as an ideal of $H_1(3)$. It is straightforward to check that $I$ is a Hopf ideal. By \eqref{eq:involutive relations 3 variables}, we see that $H_1(3)/I$ is generated by $a \coloneqq h_{13} + I$ and $b \coloneqq h_{23} + I$ with no relations, and that $a$ and $b$ are primitive elements of $H/I$. It follows that $H/I \cong U(\mathfrak{f}_2)$ as Hopf algebras, where $\mathfrak{f}_2$ is the free Lie algebra on two generators. The result follows.
\end{proof}

In the proof of Proposition \ref{prop:involutive Hopf algebra is not commutative}, an inner-faithful coaction of $U(\mathfrak{f}_2)$ is constructed. We mention this explicitly in the following corollary.

\begin{cor}\label{cor:coaction of free Lie algebra}
    The Hopf algebra $U(\mathfrak{f}_2)$ coacts on $\kk[x,y,z]$ inner-faithfully via
    \begin{align*}
        \rho \colon \kk[x,y,z] &\to \kk[x,y,z] \otimes U(\mathfrak{f}_2) \\
        x &\mapsto x \otimes 1 \\
        y &\mapsto y \otimes 1 \\
        z &\mapsto x \otimes a + y \otimes b + z \otimes 1,
    \end{align*}
    where $\mathfrak{f}_2$ is the free Lie algebra on two generators $a$ and $b$. \qed
\end{cor}

In fact, the result from Proposition \ref{prop:involutive Hopf algebra is not commutative} applies to $H_1(n)$ for arbitrary $n \geq 3$.

\begin{cor}\label{cor:involutive Hopf algebra is not commutative}
    If $n \geq 3$, then $H_1(n)$ is noncommutative, non-noetherian, and has infinite GK-dimension.
\end{cor}
\begin{proof}
    This is easily seen by noting that the map
    \begin{align*}
        \pi \colon H_1(n) &\to H_1(3) \\
        h_{ij} &\mapsto \begin{cases}
            h_{ij}, &\text{if } i,j \leq 3, \\
            1, &\text{if } i = j \geq 4, \\
            0, &\text{otherwise}.
        \end{cases}
    \end{align*}
    is a surjective homomorphism and using Proposition \ref{prop:involutive Hopf algebra is not commutative}.
\end{proof}

\begin{rem}
    Corollary \ref{cor:involutive Hopf algebra is not commutative} is reminiscent of \cite[Proposition 3.4]{Bichon}, which studies Hopf coactions on the algebra $\kk^n$. In particular, \cite[Proposition 3.4]{Bichon} proves that the universal Hopf algebra which coacts on $\kk^n$ is finite-dimensional and commutative if $n \leq 3$, but is infinite-dimensional and noncommutative if $n \geq 4$. It is interesting that a similar phenomenon occurs for involutive Hopf coactions on polynomial rings.
\end{rem}

Although $H_1(n)$ is noncommutative and involutive for $n \geq 3$, it is not clear whether $H_1(n)$ or any its of its noncommutative Hopf quotients are cosemisimple. We state this as a question.

\begin{qstn}
    Let $n \geq 3$. Is $H_1(n)$ cosemisimple? If not, are there any noncommutative cosemisimple Hopf quotients of $H_1(n)$?
\end{qstn}


\begin{thebibliography}{EGMW17}

\bibitem[AAH21]{AndruskiewitschAngionoHeckenberger}
Nicol\'as Andruskiewitsch, Iv\'an Angiono, and Istv\'an Heckenberger, \emph{On finite {GK}-dimensional {N}ichols algebras over abelian groups}, Mem. Amer. Math. Soc. \textbf{271} (2021), no.~1329, ix+125.

\bibitem[Bic08]{Bichon}
Julien Bichon, \emph{Algebraic quantum permutation groups}, Asian-Eur. J. Math. \textbf{1} (2008), no.~1, 1--13.

\bibitem[BR]{BuzagloRogalski}
Lucas Buzaglo and Daniel Rogalski, \emph{Group gradings of skew polynomial rings}, in progress.

\bibitem[CFRS14]{ChervovFalquiRubtsovSilantyev}
A.~Chervov, G.~Falqui, V.~Rubtsov, and A.~Silantyev, \emph{Algebraic properties of {M}anin matrices {II}: {$q$}-analogues and integrable systems}, Adv. in Appl. Math. \textbf{60} (2014), 25--89.

\bibitem[CKWZ16]{ChanKirkmanWaltonZhang}
Kenneth Chan, Ellen Kirkman, Chelsea Walton, and James~J. Zhang, \emph{Quantum binary polyhedral groups and their actions on quantum planes}, J. Reine Angew. Math. \textbf{719} (2016), 211--252.

\bibitem[Cra24]{Crawford}
Simon Crawford, \emph{Group coactions on two-dimensional {A}rtin-{S}chelter regular algebras}, Proc. Amer. Math. Soc. \textbf{152} (2024), no.~11, 4551--4567.

\bibitem[CRS10]{CibilsRedondoSolotar}
Claude Cibils, Mar\'ia~Julia Redondo, and Andrea Solotar, \emph{Connected gradings and the fundamental group}, Algebra Number Theory \textbf{4} (2010), no.~5, 625--648.

\bibitem[CW93]{CohenWestreich}
Miriam Cohen and Sara Westreich, \emph{Central invariants of {$H$}-module algebras}, Comm. Algebra \textbf{21} (1993), no.~8, 2859--2883.

\bibitem[CWW19]{ChirvasituWaltonWang}
Alexandru Chirvasitu, Chelsea Walton, and Xingting Wang, \emph{On quantum groups associated to a pair of preregular forms}, J. Noncommut. Geom. \textbf{13} (2019), no.~1, 115--159.

\bibitem[CWWZ14]{ChanWaltonWangZhang}
Kenneth Chan, Chelsea Walton, Yanhua Wang, and James Zhang, \emph{Hopf actions on filtered regular algebras}, J. Algebra \textbf{397} (2014), 68--90.

\bibitem[CWZ14]{ChanWaltonZhang}
Kenneth Chan, Chelsea Walton, and James Zhang, \emph{Hopf actions and {N}akayama automorphisms}, J. Algebra \textbf{409} (2014), 26--53.

\bibitem[D{\u a}s08]{Dascalescu}
Sorin D{\u a}sc{\u a}lescu, \emph{Group gradings on diagonal algebras}, Arch. Math. (Basel) \textbf{91} (2008), no.~3, 212--217.

\bibitem[EGMW17]{EtingofGoswamiMandalWalton}
Pavel Etingof, Debashish Goswami, Arnab Mandal, and Chelsea Walton, \emph{Hopf coactions on commutative algebras generated by a quadratically independent comodule}, Comm. Algebra \textbf{45} (2017), no.~8, 3410--3412.

\bibitem[EGNO15]{EGNO}
Pavel Etingof, Shlomo Gelaki, Dmitri Nikshych, and Victor Ostrik, \emph{Tensor categories}, Mathematical Surveys and Monographs, vol. 205, American Mathematical Society, Providence, RI, 2015.

\bibitem[EW14]{EtingofWalton}
Pavel Etingof and Chelsea Walton, \emph{Semisimple {H}opf actions on commutative domains}, Adv. Math. \textbf{251} (2014), 47--61.

\bibitem[GKMV24]{GoetzKirkmanMooreVashaw}
Peter Goetz, Ellen~E. Kirkman, W.~Frank Moore, and Kent~B. Vashaw, \emph{Some {A}rtin--{S}chelter regular algebras from dual reflection groups and their geometry}, 2024, arXiv:\href{https://arxiv.org/abs/2410.08959}{\texttt{2410.08959}}.

\bibitem[GZ10]{GoodearlZhang}
K.~R. Goodearl and J.~J. Zhang, \emph{Noetherian {H}opf algebra domains of {G}elfand-{K}irillov dimension two}, J. Algebra \textbf{324} (2010), no.~11, 3131--3168.

\bibitem[HNU{\etalchar{+}}24]{Hexagon}
Hongdi Huang, Van~C. Nguyen, Charlotte Ure, Kent~B. Vashaw, Padmini Veerapen, and Xingting Wang, \emph{Twisting {M}anin's universal quantum groups and comodule algebras}, Adv. Math. \textbf{445} (2024), Paper No. 109651, 55 pp.

\bibitem[KKZ10]{KirkmanKuzmanovichZhang}
E.~Kirkman, J.~Kuzmanovich, and J.~J. Zhang, \emph{Shephard-{T}odd-{C}hevalley theorem for skew polynomial rings}, Algebr. Represent. Theory \textbf{13} (2010), no.~2, 127--158.

\bibitem[KKZ17]{KirkmanKuzmanovichZhang2}
\bysame, \emph{Nakayama automorphism and rigidity of dual reflection group coactions}, J. Algebra \textbf{487} (2017), 60--92.

\bibitem[KO24]{KramerOni}
Ulrich Kr\"ahmer and Blessing~Bisola Oni, \emph{Hopf algebra (co)actions on rational functions}, Algebr. Represent. Theory \textbf{27} (2024), no.~6, 2187--2216.

\bibitem[Man88]{ManinBook}
Yu.~I. Manin, \emph{Quantum groups and noncommutative geometry}, Universit\'e de Montr\'eal, Centre de Recherches Math\'ematiques, Montreal, QC, 1988.

\bibitem[Mon93]{Montgomery}
Susan Montgomery, \emph{Hopf algebras and their actions on rings}, CBMS Regional Conference Series in Mathematics, vol.~82, Conference Board of the Mathematical Sciences, Washington, DC; by the American Mathematical Society, Providence, RI, 1993.

\bibitem[Rad12]{Radford}
David~E. Radford, \emph{Hopf algebras}, Series on Knots and Everything, vol.~49, World Scientific Publishing Co. Pte. Ltd., Hackensack, NJ, 2012.

\bibitem[RVdB17a]{RaedscheldersVanDenBergh}
Theo Raedschelders and Michel Van~den Bergh, \emph{The {M}anin {H}opf algebra of a {K}oszul {A}rtin-{S}chelter regular algebra is quasi-hereditary}, Adv. Math. \textbf{305} (2017), 601--660.

\bibitem[RVdB17b]{RaedscheldersVanDenBergh2}
\bysame, \emph{The representation theory of non-commutative {$\mathcal O({\rm GL}_2)$}}, J. Noncommut. Geom. \textbf{11} (2017), no.~3, 845--885.

\bibitem[Tak90]{Takeuchi}
Mitsuhiro Takeuchi, \emph{A two-parameter quantization of {${\rm GL}(n)$} (summary)}, Proc. Japan Acad. Ser. A Math. Sci. \textbf{66} (1990), no.~5, 112--114.

\bibitem[WW16]{WaltonWang}
Chelsea Walton and Xingting Wang, \emph{On quantum groups associated to non-{N}oetherian regular algebras of dimension 2}, Math. Z. \textbf{284} (2016), no.~1-2, 543--574.

\end{thebibliography}
\end{document}